\documentclass[12pt,reqno,a4paper]{amsart}
\usepackage{
    amsmath,  amsfonts, amssymb,  amsthm,   amscd,
    gensymb,  graphicx, comment,  etoolbox, url,
    booktabs, stackrel, mathtools,enumitem, mathdots,  microtype, lmodern,    mathrsfs, graphicx, tikz,  longtable,tabularx, float, tikz, pst-node, tikz-cd, multirow, tabularx, amscd,  bm, array, makecell, diagbox, booktabs,ragged2e, caption, subcaption }
\usepackage{makecell,slashbox}
\usepackage{xcolor}
\usepackage[all]{xy}
\usepackage{graphicx}

\usepackage[utf8]{inputenc}
\usepackage{microtype, fullpage, wrapfig,textcomp,mathrsfs,csquotes,fbb}
\usepackage[colorlinks=true, linkcolor=blue, citecolor=blue, urlcolor=blue, breaklinks=true]{hyperref}
\usepackage[capitalise]{cleveref}
\usepackage{todonotes}
\usetikzlibrary{positioning}
\usetikzlibrary{shapes,arrows.meta,calc}
\usetikzlibrary{arrows}

\newtheorem{theorem}{Theorem}[section]

\newtheorem{corollary}[theorem] {Corollary}
\newtheorem{definition}[theorem]{Definition}
\newtheorem{example}[theorem]{Example}
\newtheorem{lemma}[theorem]{Lemma}

\newtheorem{proposition}[theorem]{Proposition}
\newtheorem{remark}[theorem]{Remark}

\newcommand{\TC}{\mathrm{TC}}

\newcommand{\ct}{\mathrm{cat}}

\newcommand{\sct}{\mathrm{secat}}

\newcolumntype{x}[1]{>{\centering\arraybackslash}p{#1}}

\begin{document}
\title[]{Parametrized homotopic distance}
\author[N. Daundkar]{Navnath Daundkar}
\address{Department of Mathematics, Indian Institute of Science Education and Research Pune, India.}
\email{navnath.daundkar@acads.iiserpune.ac.in}
\author[]{J.M. Garc\'ia-Calcines}
\address{Departamento de Matem\'aticas, Estad\'istica e Investigaci\'on Operativa, Universidad de La Laguna, Avenida Astrof\'isico Francisco S\'anchez S/N, 38200 La Laguna, Spain.}
\email{jmgarcal@ull.edu.es}

\thanks{}

\begin{abstract} 
We introduce the concept of parametrized homotopic distance, extending the classical notion of homotopic distance to the fibrewise setting. We establish its correspondence with the fibrewise sectional category of a specific fibrewise fibration and derive cohomological lower bounds and connectivity upper bounds under mild conditions. We also analyze the behavior of parametrized homotopic distance under compositions and products of fibrewise maps, along with its interaction with the triangle inequality. 

We establish several sufficient conditions for fibrewise $H$-spaces to admit a fibrewise division map and prove that their parametrized topological complexity equals their fibrewise unpointed LS category, extending Lupton and Scherer's theorem to the fibrewise setting. Additionally, we give sharp estimates for the parametrized topological complexity of a class fibrewise $H$-spaces which arises as sphere bundles with fibre  $S^7$.  Furthermore, we estimate the parametrized homotopic distance of fibre-preserving, fibrewise maps between fibrewise fibrations, in terms of the parametrized homotopic distance of the induced fibrewise maps between individual fibres, as well as the fibrewise unpointed Lusternik–Schnirelman category of the base space. 

Finally, we define and study a pointed version of parametrized homotopic distance, establishing cohomological bounds and identifying key conditions for its equivalence with the unpointed version, thus providing a finer classification of fibrewise homotopy invariants.
\end{abstract}

\maketitle

\section{Introduction}
For a topological space $X$, the study of classical homotopy invariants, such as the Lusternik–Schnirelman category $\ct(X)$, introduced by Lusternik and Schnirelman \cite{L-S-cat}, and the topological complexity $\TC(X)$, introduced by Farber \cite{F}, has led to significant insights into the connections between algebraic topology and applications in fields like critical point theory and robot motion planning. These invariants exhibit structural similarities. This has motivated further exploration of their relationships and possible generalizations.

In \cite{macias2022homotopic}, Macías-Virgós and Mosquera-Lois introduced the concept of homotopic distance between continuous maps, unifying and extending both $\ct(X)$ and $\TC(X)$ within a common framework. 
The homotopic distance $D(f,g)$ between two continuous maps $f,g:X\to Y$ is the smallest integer $n\geq 0$ (or infinity if no such $n$ exists) for which $X$ admits a cover by $n+1$ open sets, on each of which $f$ and $g$ are homotopic. For instance, under the assumption of path-connectedness, $\ct(X)=D(id_X,c)$, where $id_X$ is the identity map and $c$ is a constant map defined on $X$. Furthermore, for the projections $pr_1,pr_2:X\times X\to X$, we have $\TC(X)=D(pr_1,pr_2)$ (see \cite[Proposition 2.6]{macias2022homotopic}). 
This framework not only simplifies proofs of key properties of $\ct(X)$ and $\TC(X)$, but also leads to new results, including refined inequalities that serve as lower bounds for these invariants. Specifically, Macías-Virgós and Mosquera-Lois studied the behavior of homotopic distance under compositions and products, deriving the well-known product inequalities for both the Lusternik–Schnirelman category and topological complexity. They also explored the homotopic distance of maps defined on $H$-spaces, recovering the result of Lupton and Scherer \cite{L-S}, which asserts that the topological complexity of $H$-spaces coincides with their Lusternik–Schnirelman category. Moreover, they established a cohomological lower bound for the homotopic distance and examined the homotopic distance of fiber-preserving maps between fibrations.

In many real-world scenarios, the complexity of motion planning is influenced by additional parameters, motivating the study of parametrized settings. To incorporate such dependencies, Cohen, Farber, and Weinberger developed in \cite{C-F-W} a natural topological framework for parameterized settings. These scenarios are modeled by a fibration $p:E\to B$, where $E$ is viewed as a union of fibers $X_b=p^{-1}(b)$, indexed by points $b\in B$. Selecting a point $b\in B$ specifies an external condition of the system. The parameterized topological complexity $\TC[p:E\to B]$ quantifies the complexity of universal motion planning algorithms in this context. Foundational results for $\TC[p:E\to B]$ provide lower and upper bounds based on the topology of $E$, $B$ and the fibers $X_b$.

García-Calcines introduced a more general framework in \cite{Calcines-fibrewiseTC} for a fibrewise space $X$ over $B$, which need not necessarily be a fibration. In this context, the fibrewise topological complexity $\TC_B(X)$ coincides with the version introduced by Cohen, Farber, and Weinberger when $p$ is a fibration. García-Calcines also demonstrated in \cite[Proposition 11]{Calcines-fibrewiseTC} that fibrewise topological complexity is a fibrewise homotopy invariant, and introduced the pointed version of parameterized topological complexity, along with sufficient conditions under which the two notions coincide (see \cite[Proposition 29]{Calcines-fibrewiseTC}).

This paper introduces the \textit{parametrized homotopic distance}, a novel extension of homotopic distance to the fibrewise setting, providing a unified framework for various fibrewise homotopy invariants.  Specifically, the parametrized (or fibrewise) homotopic distance $D_B(f,g)$ between fibrewise maps $f,g:X\to Y$, where $X, Y$ are fibrewise spaces over $B$, is defined as the smallest integer $n\geq 0$ (or infinity) such that $X$ can be covered by $n+1$ open sets where $f$ and $g$ are fibrewise homotopic on each set. The parametrized homotopic distance provides a unifying framework that generalizes previous notions, including the fibrewise unpointed Lusternik–Schnirelman category and parametrized topological complexity. More specifically, if $X$ is a fibrewise space over $B$ and $pr_1,pr_2:X\times_B X\to X$ are the corresponding projections, then we have $\TC_B(X)=D_B(pr_1,pr_2)$ (see \Cref{cor:parametrizedTC-as-phd}). Moreover, if $X$ is a fibrewise pointed space over $B$, then $\ct_B^*(X)=D_B(id_X,s_X\circ p_X)$.
By relating parametrized homotopic distance to the fibrewise sectional category in the sense of \cite{GC}, of specific fibrations, we derive cohomological lower bounds and connectivity upper bounds. Additionally, we explore its behaviour under compositions, products, and fibre-preserving maps. These results not only deepen our understanding of fibrewise spaces but also provide new algebraic and topological tools with direct applications in motion planning and the study of fibrewise $H$-spaces.

The paper is organized as follows. In \Cref{sec:prelim-fib-homotopy-theory}, we review basic concepts of fibrewise homotopy theory. \Cref{sec:phd} introduces the parametrized homotopic distance and establishes its connection with the fibrewise sectional category, leading to cohomological and dimension-connectivity bounds. \Cref{sec:phd-properties} explores its fundamental properties, including invariance, composition rules, and triangle inequalities. \Cref{sec: fibrewise-H-spaces} focuses on fibrewise $H$-spaces, proving that those admitting a division map satisfy $\TC_B(X) = \ct_B^*(X)$ and analyzing further implications. \Cref{sec: fibrewise fibrations} examines parametrized homotopic distance in the context of fibrewise fibrations, relating it to the LS category of the base space. \Cref{sec:pphd} introduces and studies the pointed version of the concept, providing cohomological bounds and conditions under which it coincides with the unpointed version. Finally, \Cref{sec:pphd-vs-phd} compares both versions, establishing criteria for their equivalence.

\section{Fibrewise homotopy theory}\label{sec:prelim-fib-homotopy-theory}

We begin by reviewing key notations and results from fibrewise homotopy theory. For further details, the reader is referred to \cite{C-J} and \cite{GC}.

Let $B$ be a fixed topological space. A \emph{fibrewise space} over $B$ is defined as a pair $(X, p_X)$, where $X$ is a topological space, and $p_X:X\rightarrow B$ is a map, commonly referred to as the projection of the fibrewise space. When there is no risk of confusion, the pair $(X, p_X)$ will simply be denoted by $X$ and referred to as a fibrewise space. Given two fibrewise spaces $X$ and $Y$, a fibrewise map (over $B$) from $X$ to $Y$ is a map $f: X \rightarrow Y$ that satisfies $p_Y \circ f = p_X$.

We denote the category of fibrewise spaces and fibrewise maps over $B$ by $\mathbf{Top}_B$. In this category, the space $B$ with the identity map serves as the final object, while the initial object is the empty set $\emptyset$, with a unique projection map to $B$. If $X$ and $Y$ are two fibrewise spaces, their fibrewise product is given by $$X\times _BY=\{(x,y)\in
X\times Y:p_X(x)=p_Y(y)\}$$ which is the pullback of the maps $p_X$ and $p_Y$. This construction provides the categorical product of $X$ and $Y$ in $\mathbf{Top}_B$.

Let $I$ represent the closed unit interval $[0,1]$ with the standard topology inherited from $\mathbb{R}$. The \emph{fibrewise cylinder} of a fibrewise space $X$ is the product space $X \times I$, with the projection given by the composition ${X \times I} \stackrel{pr_1}{\longrightarrow} X \stackrel{p_X}{\longrightarrow} B$. We refer to the fibrewise cylinder of $X$ as $I_B(X)$. The concepts of \emph{fibrewise homotopy} $\simeq _B$ between fibrewise maps and \emph{fibrewise homotopy equivalence} follow naturally from this construction.

If $X$ is a fibrewise space consider the pullback in
the category $\mathbf{Top}$ of topological spaces and maps:
\begin{equation}\label{eq:PBX}
 \xymatrix{
{P_B(X)} \ar[r] \ar[d] & {X^I} \ar[d]^{p_X^I} \\
{B} \ar[r]_c & {B^I.} }   
\end{equation}

Here $X^I$ (and $B^I$) denotes the free path-space provided with the compact-open topology and $p_X^I$ is the obvious map induced
by precomposing with $p_X$. Besides $c:B\rightarrow B^I$ is the map that carries any $b\in B$ to the
constant path $c_b$ in $B^I$. Thus, $P_B(X)$ has the expression
$$P_B(X)=B\times _{B^I}X^I=\{(b,\alpha )\in B\times X^I:c_b=p_X\circ \alpha \}$$
\noindent with projection $P_B(X)\rightarrow B$, $(b,\alpha )\mapsto b,$ the base change of $p_X^I$ in this pullback. The space $P_B(X)$ is called the \emph{fibrewise cocylinder} of $X$, or equivalently, the
\emph{fibrewise free path space} of $X$.

\begin{remark}\label{prim}
Observe that $P_B(X)$ can also be described as the space of all paths $\alpha :I\rightarrow X$ such
that the path $p_X\circ \alpha $ is constant, i.e., paths lying in a single fibre of $X.$
This description is provided in \cite{C-F-W}, where the notation used is $X_B^I$ for the fibrewise cocylinder.
Also note that $P_B(X)$ is fibrewise homotopy equivalent to $X$. Indeed, the fibrewise map
$\gamma _X:X\rightarrow P_B(X),\hspace{8pt}x\mapsto (p_X(x),c_x)$,
is a fibrewise homotopy equivalence with a homotopy inverse $\gamma '_X:P_B(X)\rightarrow X$ defined by
$\gamma '_X(b,\alpha ):=\alpha (0)$.
\end{remark}

The fibrewise cylinder and fibrewise cocylinder constructions give rise to functors
$I_B,P_B:\mathbf{Top}_B\rightarrow \mathbf{Top}_B.$
Associated with the functor $I_B$ there are straightforwardly defined natural
transformations $i_0,i_1:X\rightarrow I_B(X)$ and $\rho
:I_B(X)\rightarrow X$. Similarly, associated
with $P_B$ there are natural transformations
$d_0,d_1:P_B(X)\rightarrow X$ and $c:X\rightarrow P_B(X)$. Moreover, $(I_B,P_B)$ is an
\emph{adjoint pair} in the sense of Baues (see \cite[p.29]{B}).
A \emph{fibrewise fibration} is a fibrewise map
$p:E\rightarrow Y$ satisfying the Homotopy Lifting
Property with respect to any fibrewise space, i.e.,
given any commutative diagram of solid arrows in $\mathbf{Top}_B$
$$\xymatrix{
{Z} \ar[r]^f \ar[d]_{i_0} & {E} \ar[d]^p \\
{I_B(Z)} \ar[r]_H \ar@{.>}[ur] & {Y} }$$
\noindent the dotted arrow exists in $\mathbf{Top}_B$ making the entire diagram commutative.
As the functor $I_B$ is left adjoint to the functor $P_B$ it is easy to check that, actually,
fibrewise fibrations are precisely the internal fibrations in $\mathbf{Top}_B$ with respect to $P_B.$
Therefore, $\mathbf{Top}_B$ together with $P_B$ the fibrewise cocylinder is a $P$-category in the sense of Baues \cite[p.31, Prop (4.6)]{B}.

If $p:E\rightarrow Y$ is any fibrewise map such that it is an ordinary
Hurewicz fibration, then $p$ is a fibrewise fibration. In
general, the converse is not true. For instance, if $X$ is a
fibrewise space, then $p_X:X\rightarrow B$ is always a fibrewise
fibration, but $p_X$ need not be a Hurewicz fibration.

\begin{remark}
From the general axiomatic theory of a $P$-category, if $X$ is any fibrewise space, then the fibrewise map
$\Pi =(d_0,d_1):P_B(X)\rightarrow X\times _B X,$ defined by $(b,\alpha )\mapsto (\alpha (0),\alpha (1))$,
is always a fibrewise fibration, which is not necessarily a Hurewicz fibration.
Nevertheless, we point out that when the projection $p_X:X\rightarrow B$ is a Hurewicz fibration, one can check that $\Pi :P_B(X)\rightarrow X\times _B X$ is also a Hurewicz fibration. This can be checked by simply considering the co-gluing theorem for Hurewicz fibrations in $\mathbf{Top}$ (see, for instance, the dual part of \cite[Chapter II, Lemma 1.2 (a)]{B}).
\end{remark}

A fibrewise map $j:A\rightarrow X$ over $B$ is said to be a \emph{fibrewise
cofibration} if it satisfies the Homotopy Extension Property. That is, for
any fibrewise map $f:X\rightarrow
Y$ and any fibrewise homotopy $H:I_B(A)\rightarrow Y$ such that
$H\circ i_0=f\circ j,$ there exists a fibrewise homotopy
$\widetilde{H}:I_B(X)\rightarrow Y$ such that $\widetilde{H}\circ i_0=f$
and $\widetilde{H}\circ I_B(j)=H$
$$\xymatrix{
{A} \ar[d]_j \ar[r]^{i_0} & {I_B(A)} \ar[d]^H
\ar@/^2pc/[ddr]^{I_B(j)} & \\ {X} \ar[r]^f \ar@/_2pc/[drr]_{i_0} &
{Y} &
\\ & & {I_B(X)} \ar@{.>}_{\widetilde{H}}[ul]}$$

As known, fibrewise cofibrations are cofibrations in the usual
sense. Therefore, without loss of generality, we can consider, fibrewise cofibrations as inclusions $A\hookrightarrow X.$
The pair $(X,A)$ is then called \emph{fibrewise cofibred pair}. Similarly, fibrewise cofibrations are precisely
the internal cofibrations in $\mathbf{Top}_B$ with respect to $I_B.$
Hence, $\mathbf{Top}_B$ together with $I_B$, the fibrewise cylinder, is an $I$-category in the sense of Baues \cite[p.31, Prop (4.6)]{B}.

In fact, if $fib_B,$ $\overline{cof}_B$ and $he_B$ denote the classes of
fibrewise fibrations, closed fibrewise cofibrations (equivalently,
closed fibrewise cofibred pairs), and fibrewise homotopy
equivalences, respectively, then the category $\mathbf{Top}_B$ together with the classes of
$\overline{cof}_B,$ $fib_B$ and $he_B$ has an IP-category
structure in the sense of Baues \cite{B}. Moreover,
the category $\mathbf{Top}_B$ with
the classes $\overline{cof}_B,$ $fib_B$ and $he_B$ has
a proper closed model category structure in the sense of Quillen \cite{GC,M-S}.

A \emph{fibrewise pointed space} over $B$ is a triple $(s_X, X, p_X)$, where $(X, p_X)$ is a fibrewise space and $s_X: B \to X$ is a section of $p_X$, meaning $p_X \circ s_X = id_B$. For simplicity, we will refer to the fibrewise pointed space $(s_X, X, p_X)$ as $X$, unless clarity requires otherwise.  
Given fibrewise pointed spaces $X$ and $Y$, a fibrewise pointed map $f: X \to Y$ is a fibrewise map such that $f \circ s_X = s_Y$.

The category of fibrewise pointed spaces and fibrewise pointed maps will be denoted by $\mathbf{Top}(B)$. Note that the space $B$, together with the identity map, serves as the zero object, making $\mathbf{Top}(B)$ a pointed category. Any subspace $A \subseteq X$ containing the section (i.e., $s_X(B) \subseteq A$) is a fibrewise pointed space. In this case, the inclusion $A \hookrightarrow X$ is a fibrewise pointed map. Such subspaces are called fibrewise pointed subsets of $X$.

For any fibrewise pointed space $X$, we can define its \emph{fibrewise pointed cylinder} as the pushout:
$$
\xymatrix{
{B \times I} \ar[r]^{pr} \ar[d]_{s_X \times id} & {B} \ar[d] \\
{X \times I} \ar[r]  & {I_B^B(X)}. }
$$
The projection is naturally induced by the pushout property. This pointed cylinder functor gives rise to the notion of \emph{fibrewise pointed homotopy} between fibrewise pointed maps, denoted by $\simeq_B^B$. A fibrewise pointed homotopy $F: I_B^B(X) \to Y$ is equivalent to a fibrewise homotopy $F': I_B(X) \to Y$ satisfying $F'(s_X(b), t) = s_X(b)$ for all $b \in B$ and $t \in I$. The notion of \emph{fibrewise pointed homotopy equivalence} follows naturally.

We can also define
$P_B^B(X) = B \times_{B^I} X^I = \{(b, \alpha) \in B \times X^I : c_b = p_X\circ \alpha\},$ which is the fibrewise space $P_B(X)$ together with the section $(id_B, c\circ s_X): B \to P_B^B(X)$, induced by the pullback property. There are functors
$I_B^B, P_B^B: \mathbf{Top}(B) \to \mathbf{Top}(B),$
along with natural transformations $i_0, i_1: X \to I_B^B(X)$, $\rho: I_B^B(X) \to X$, and $d_0, d_1: P_B^B(X) \to X$, $c: X \to P_B^B(X)$. Moreover, $(I_B^B, P_B^B)$ forms an adjoint pair in the sense of Baues.

Using these functors, we define the notions of \emph{(closed) fibrewise pointed cofibration} and \emph{fibrewise pointed fibration}, characterized by the usual Homotopy Extension Property and Homotopy Lifting Property in $\mathbf{Top}(B)$, respectively. Furthermore, $\mathbf{Top}(B)$ has both $I$-category and $P$-category structures as defined by Baues (\cite[p.31]{B}).

Every fibrewise map that is a fibrewise cofibration is also a fibrewise pointed cofibration. Similarly, any fibrewise fibration is a fibrewise pointed fibration. However, as noted by May and Sigurdsson in \cite[p.82]{M-S}, the converses are not generally true, even in the simple case where $B$ is a point.  

Nevertheless, we can consider a suitable subcategory of $\mathbf{Top}(B)$ where the notions of fibrewise and fibrewise pointed cofibrations, fibrewise and fibrewise pointed fibrations, as well as fibrewise and fibrewise pointed homotopy equivalences coincide. This enables a more unified treatment of these concepts within the chosen subcategory.

A \emph{fibrewise well-pointed} space is a fibrewise pointed space $X$ such that the section $s_X: B \to X$ is a closed fibrewise cofibration. Let $\mathbf{Top}_w(B)$ denote the full subcategory of $\mathbf{Top}(B)$ consisting of fibrewise well-pointed spaces. 

\begin{proposition}\label{important1}\cite[Proposition 3.3]{GC}  Let $f: X \to Y$ be a fibrewise pointed map between fibrewise well-pointed spaces over $B$. Then:  
\begin{enumerate}  
\item[(i)] $f$ is a fibrewise pointed fibration if and only if it is a fibrewise fibration.  
\item[(ii)] If $f$ is a closed map, then $f$ is a fibrewise pointed cofibration if and only if it is a fibrewise cofibration.  
\item[(iii)] $f$ is a fibrewise pointed homotopy equivalence if and only if it is a fibrewise homotopy equivalence.  
\end{enumerate}  
\end{proposition}

Although $\mathbf{Top}_w(B)$ is not closed under finite limits and colimits and, therefore, cannot form a model category, the following result suffices for our homotopical framework. 

\begin{proposition}\label{important2}\cite[Proposition 3.2]{GC}  The category $\mathbf{Top}_w(B)$ is closed under pullbacks of fibrewise pointed maps that are fibrewise fibrations. Likewise, it is closed under pushouts of fibrewise pointed maps that are closed fibrewise cofibrations.  
\end{proposition}

Taking the two propositions above into account, it follows that $\mathbf{Top}_B$ induces cofibration and fibration category structures on $\mathbf{Top}_w(B)$ in the sense of Baues.

\section{Introducing parametrized homotopic distance}\label{sec:phd}

We here introduce the central concept of our study, along with its formal definition and the key initial properties that will be explored in this section.

\begin{definition}
Let $f,g:X\to Y$ be two fibrewise maps between fibrewise spaces $X$ and $Y$ over $B$. The parametrized homotopic distance between $f$ and $g$, denoted $D_B(f,g)$, is defined as the smallest integer $n$ for which there exists an open cover $\{U_0,\dots, U_n\}$ of $X$ such that $f|_{U_i}\simeq_{B}g|_{U_i}$ for $0\leq i\leq n$.  If no such open cover exists, we set $D_B(f,g)=\infty$. 
\end{definition}
The parametrized homotopic distance serves as a measure of how close two fibrewise maps are to being fibrewise homotopic. Moreover, given two fibrewise maps $f, g: X \rightarrow Y$, it is clear that, in general,  $D(f, g) \leq D_B(f, g)$. This inequality can, of course, be strict.

From the definition above, we immediately obtain the following properties:
\begin{enumerate}
\item[(i)] $D_B(f,g)=D_B(g,f);$ 

\item[(ii)] $D_B(f,g)=0$ if, and only if, $f\simeq _B g$;

\item[(iii)] If $f\simeq _B g$ and $f'\simeq _B g'$, then $D_B(f,g)=D_B(f',g').$
    
\end{enumerate}

Our goal is to establish a connection between the parametrized homotopic distance and the fibrewise sectional category. To this end, we first recall the notion of fibrewise sectional category, originally introduced and extensively developed in \cite{GC}. Here, we adopt this framework and utilize several of its key properties to achieve our objectives.

\begin{definition}
The fibrewise sectional category of a fibrewise map $f: E \rightarrow X$, denoted $\mbox{secat}_B(f)$, is the smallest integer $n \geq 0$ such that $X$ admits an open cover $\{U_i\}_{i=0}^n$ where each $U_i$ has a fibrewise homotopy section $s_i: U_i \rightarrow E$ of $f$, meaning that $f \circ s_i \simeq_B \mbox{inc}_{U_i}$:
$$
\xymatrix{
    {U_i} \ar@{^(->}[rr]^{\mbox{inc}_{U_i}} \ar[dr]_{s_i} & & {X} \\
    & {E}. \ar[ur]_f &
}
$$
\noindent If no such $n$ exists, then we set $\mbox{secat}_B(f) = \infty$.
\end{definition}

\begin{remark} As in the classical case, when $f: E \rightarrow X$ is a fibrewise fibration, the definition of fibrewise sectional category can be refined by requiring the triangles to be strictly commutative, rather than merely fibrewise homotopy commutative.
\end{remark}

Given two fibrewise maps $f,g:X\rightarrow Y,$ we can consider the following pullback square:
\begin{equation}\label{dia: pullback-fibrewise-fib}
 \xymatrix{
{\mathcal{P}_B(f,g)} \ar[d]_{\widetilde{\Pi }_Y} \ar[r] & {P_B(Y)} \ar[d]^{\Pi _Y} \\
{X} \ar[r]_{(f ,g )} & {Y\times _BY.} }   
\end{equation}
Note that $\widetilde{\Pi }_Y$ is a fibrewise fibration, as $\Pi_Y$ is, and they are connected by a pullback.
\begin{proposition}\label{prop:phd-equals-secatB}
$D_B(f,g)=\mbox{secat}_B(\widetilde{\Pi} _Y)$  
\end{proposition}

\begin{proof}
Let $U \subset X$ be an open subset with a fibrewise homotopy $F: f|_{U} \simeq_B g|_{U}$. Denote by $s: U \to P_B(Y)$ the adjoint map of $F$. By the universal property of the pullback, we obtain the following induced diagram:
$$
\xymatrix{
U \ar@/^1pc/[drr]^{s} \ar@/_1pc/[ddr]_{inc_U} \ar@{.>}[dr]^{\sigma } & & \\
 & \mathcal{P}_B(f,g) \ar[r] \ar[d]_{\widetilde{\Pi}_Y} & P_B(Y) \ar[d]^{\Pi_Y} \\
 & X \ar[r]_{(f,g)} & Y \times_B Y.
}
$$
\noindent This diagram shows that $\sigma $ provides a local section of the fibrewise map $\widetilde{\Pi}_Y$ over $U$. Conversely, given such a local section $\sigma: U \to \mathcal{P}_B(f,g)$, the composition of $\sigma$ with the fibrewise map $\mathcal{P}_B(f,g) \to P_B(Y)$ yields a fibrewise homotopy $f|_{U} \simeq_B g|_{U}$. By applying this construction to open covers, we obtain the result.
\end{proof}

\begin{corollary}\label{cor:parametrizedTC-as-phd}
Let $X$ be a fibrewise space. Denote by $pr_1,pr_2:X\times _BX\rightarrow X$ the projection maps corresponding to the first and second factors in the fibrewise product space $X\times _BX$.  Then, 
$$\TC_B(X)=D_B(pr_1,pr_2).$$
\end{corollary}

\begin{proof}
Just observe that the fibrewise map $(pr_1,pr_2)$ is the identity on $X\times _BX$ and, consequently, $\widetilde{\Pi} _X=\Pi _X$.  
\end{proof}

Next, we will examine an interesting connection between the parametrized homotopy distance and the fibrewise unpointed LS category, as defined by N. Iwase and M. Sakai in \cite{iwase2010topological}. To begin, we will recall its definition.

\begin{definition}
Let $X$ be a fibrewise pointed space over $B$. The fibrewise unpointed LS category of $X$, denoted $\ct_B^{*}(X)$, is defined as the smallest non-negative integer $n$ for which there exists an open cover $\{U_0,\dots, U_n\}$ of $X$ such that, for each $0\leq i\leq n$, the following diagram commutes up to fibrewise homotopy:
$$
\xymatrix{
    {U_i} \ar@{^(->}[rr]^{\mbox{inc}_{U_i}} \ar[dr]_{p_X|_{U_i}} & & {X} \\
    & {B}. \ar[ur]_{s_X} &
}
$$
If no such integer exists, we set $\ct_B^{*}(X)=\infty .$
\end{definition}

\begin{remark}
Observe that $\ct_B^{*}(X)=\mbox{secat}_B(s_X)$ corresponds to the fibrewise sectional category of the section $s_X:B\rightarrow X$, where $s_X$ is regarded as a fibrewise map. Moreover, by the definition of the parametrized homotopic distance, we also have $\ct_B^{*}(X)=D_B(id_X,s_X\circ p_X).$
\end{remark}

Suppose $X$ is a fibrewise pointed space over $B$. Then, $X\times _BX$ is also  fibrewise pointed with projection $p_{X\times _BX}=p_X\circ pr_2$ and section $s_{X\times _BX}=(s_X,s_X)$, that is, $s_{X\times _B X}(b)=(s_X(b),s_X(b)),$ for all $b\in B$.
Consider the fibrewise maps $i_1, i_2:X\to X\times _BX$ defined by 
$$i_1(x) := (x, (s_X \circ p_X)(x)) \quad \text{and} \quad i_2(x) := ((s_X \circ p_X)(x), x)$$ \noindent which can be written more compactly as $i_1=(id_X,s_X\circ p_X)$ and $i_2=(s_X\circ p_X,id_X)$. Then we have the following result:

\begin{proposition}\label{prop:fLScat=DBi1i2}
Let $X$ be a fibrewise pointed space. Then
$\ct_B^{*}(X)=D_B(i_1,i_2)$.
\end{proposition}

\begin{proof}
Let $U$ be a fibrewise categorical subset of $X$. Then there exists a fibrewise homotopy $H: I_B(U)\to X$ such that $H(u,0)=u$ and $H(u,1)=(s_X\circ p_X)(u)$, for all $u\in U$. 
We can now define the fibrewise homotopy $H': I_B(U)\to X\times_B X$ between $i_{1}|_U$ and $i_{2}|_U$ by 
\[H'(u,t)=\begin{cases}
    (H(u,2t), s_X\circ p_X(u)), & 0\leq t\leq \frac{1}{2},\\
    (s_X\circ p_X(u), H(u, 2-2t)), & \frac{1}{2}\leq t\leq 1.
\end{cases}\]

Now, consider a fibrewise homotopy $F: I_B(U)\to X\times_B X$ between $i_{1}|_U$
and $i_{2}|_U$. Then, composing $F$ with the first projection gives a fibrewise homotopy between the inclusion $i|_{U}:U\hookrightarrow X$ and $s_X\circ p_{X}|_U:U\to X$.
\end{proof}

We now establish a cohomological lower bound and a homotopy dimension-connectivity upper bound for the parametrized homotopic distance. Recall that a fibrewise space $X$ is called \emph{fibrant} if the projection $p_X : X \to B$ is a Hurewicz fibration. As previously commented, in this case, the fibrewise map $\Pi_X : P_B(X) \to X \times_B X$ is also a Hurewicz fibration.

\begin{theorem}\label{thm:lower-upper-bounds}
Let $f, g : X \to Y$ be fibrewise maps with $X$ and $Y$ fibrant spaces. Then the following hold:
\begin{enumerate}
    \item Let $z_1, \dots, z_k \in H^*(Y \times_B Y; R)$, where $R$ is any coefficient ring, such that $\Delta_Y^*(z_i) = 0$ for $1 \leq i \leq k$ and $(f, g)^*(z_1 \cdots z_k) \neq 0$. Then $D_B(f, g) \geq k.$
    
    \item Suppose $Y$ is path-connected and $X$ has the homotopy type of a CW-complex. If $p_Y : Y \to B$ is an $r$-equivalence for some $r\geq 1$, then
    $
    D_B(f, g) \leq \mbox{hdim}(X)/r,
    $
    where $\mbox{hdim}(X)$ denotes the homotopy dimension of $X$, defined as the smallest dimension of any CW-complex that has the homotopy type of $X$.
\end{enumerate}
\end{theorem}

\begin{proof}
Consider the fibrewise map $\widetilde{\Pi}_Y:\mathcal{P}_B(f,g)\rightarrow X.$
By \Cref{prop:phd-equals-secatB}, we have $D_B(f, g) = \sct_B(\widetilde{\Pi}_Y)$.
Since $Y$ is fibrant, the map $\Pi_Y: P_B(Y) \rightarrow Y \times_B Y$ is a Hurewicz fibration, which implies that $\widetilde{\Pi}_Y$ is also a Hurewicz fibration. Moreover, since $X$ is fibrant, $\mathcal{P}_B(f, g)$ is fibrant as well. Thus, $\widetilde{\Pi}_Y$ is a fibrewise map between fibrant spaces over $B$.

The equality $\sct_B(\widetilde{\Pi}_Y) = \sct(\widetilde{\Pi}_Y)$ (i.e., the fibrewise sectional category equals the ordinary sectional category of $\widetilde{\Pi}_Y$) follows from \cite[Theorem 2.10]{GC}. Since $\Pi_Y$ is a Hurewicz fibration, \cite[Corollary 1.5]{GCrelsecat} implies that $\sct(\widetilde{\Pi}_Y) = \sct_{(f, g)}(\Pi_Y)$, the relative sectional category of $\Pi_Y$ with respect to the map $(f,g):X\rightarrow Y\times _BY.$

Combining these results with the fact that $\Delta_Y^*(z_i) = 0$ if and only if $(\Pi_Y)^*(z_i) = 0$ for cohomology classes $z_i \in H^*(Y \times_B Y; R)$, we obtain the desired inequality by \cite[Proposition 3.1(1)]{GCrelsecat} (see also \cite[Proposition 3.8 (5)]{Gonzalez-Grant-Vandebroucq}):
$
D_B(f, g) = \sct_{(f, g)}(\Pi_Y) \geq k,
$
thus completing the proof of our first assertion.

Now observe that the assumption that $p_Y : Y \to B$ is an $r$-equivalence implies that $\Pi_Y$ is an $(r-1)$-equivalence. This yields the desired inequality by \cite[Proposition 3.1(2)]{GCrelsecat} or \cite[Proposition 3.8 (4)]{Gonzalez-Grant-Vandebroucq}.
\end{proof}

\begin{example}
 Let $X$ be a fibrant fibrewise space over $B$. By Corollary \ref{cor:parametrizedTC-as-phd}, we have $D_B(f,g)=\TC_B(X)$, where $pr_1, pr_2:X\times_B X\to X$ denote the canonical projections. Note that $X\times_B X$ is also fibrant. Since $(pr_1,pr_2)^*=id_{X\times_B X}^*=id_{H^{*}(X\times_B X;R)}$, the lower bound in \Cref{thm:lower-upper-bounds} aligns with the cohomological lower bound established in \cite[Proposition 7.3]{C-F-W}.
 
 Similarly, if $p_X:X\rightarrow B$ is a locally trivial fibration with an $(r-1)$-connected fiber $F$ (i.e., $p_X$ is an $r$-equivalence), where $X$, $B$ and $F$ are CW-complexes, the upper bound derived in \Cref{thm:lower-upper-bounds} coincides with the upper bound presented in \cite[Proposition 7.2]{C-F-W}.
\end{example}

\section{Properties of parametrized homotopic distance}\label{sec:phd-properties}

This section focuses on analyzing the fundamental properties of the parametrized homotopic distance. We will begin by examining its behavior with respect to the composition of fibrewise maps and will also explore some characteristics derived from the fibrewise versions of the Lusternik-Schnirelman category and topological complexity.

\begin{proposition}\label{prop:composition-inequalities}
Let $f,g:X\rightarrow Y$ be fibrewise maps.
\begin{enumerate}
    \item If $h:Y\rightarrow Z$ is a fibrewise map, then $D_B(h\circ f,h\circ g)\leq D_B(f,g).$ 
    \item If $k:Z\rightarrow X$ is a fibrewise map, then
    $D_B(f\circ k,g\circ k)\leq D_B(f,g).$ 
\end{enumerate}    
\end{proposition}
\begin{proof}
To prove part $(1)$, consider an open set $U \subset X$ with a fibrewise homotopy $F: f|_U \simeq_B g|_U$. Composing with $h$, we obtain $h \circ F$, which gives $(h \circ f)|_U \simeq_B (h \circ g)|_U$. Applying this construction to an open cover yields the desired inequality.  
For part $(2)$, let $V := k^{-1}(U) \subset Z$, and denote $\bar{k}: V \to U$ as the restriction of $k$. The induced map $\tilde{k}: I_B(V) \to I_B(U)$ allows us to transfer the homotopy $F$, yielding $(f \circ k)|_V \simeq_B (g \circ k)|_V$. Again, applying this argument to an open cover gives the required inequality.
\end{proof}

Given any fibrewise map $f:X\rightarrow Y$, where $Y$ is fibrewise pointed, the \emph{fibrewise unpointed LS category of $f$}, denoted by $\ct_B^{*}(f)$, is the smallest non-negative integer $n$ (or infinity if such $n$ does not exist) for which there exists an open cover $\{U_0,\dots, U_n\}$ of $X$ such that $f|_{U_i}\simeq _B s_Y\circ p_X|_{U_i}$, for each $0\leq i\leq n$. 
It is immediate that if $f=id_X$, we recover the notion of  fibrewise unpointed LS category of $X$. Moreover, it follows that $\ct_B^{*}(f)=D_B(f,s_Y\circ p_X).$

\begin{corollary}\label{cor: rel-catBf-catBX-catBY}
Let $f:X\rightarrow Y$ be a fibrewise pointed map. Then
$$\ct_B^{*}(f)\leq \mbox{min}\{ \ct_B^{*}(X), \ct_B^{*}(Y)\}.$$
\end{corollary}

\begin{proof}
The following inequality follows from part $(1)$ of \Cref{prop:composition-inequalities}: 
\[D_B(f\circ id_X, f\circ s_X\circ p_X)\leq D_B(id_X,s_X\circ p_X).\]
Since $f\circ s_X=s_Y$, we have $D_B(f\circ id_X, f\circ s_X\circ p_X)=\ct_B^{*}(f)$. Furthermore, $D_B(id_X,s_X\circ p_X)=\ct_B^{*}(X)$, so we obtain the inequality $\ct_B^{*}(f)\leq  \ct_B^{*}(X)$.
Similarly, by using part $(2)$ of \Cref{prop:composition-inequalities}, we get the inequality
$D_B(id_Y\circ f,s_Y\circ p_Y\circ f)\leq D_B(id_Y,s_Y\circ p_Y).$
Since $p_Y\circ f=p_X$, we obtain $D_B(id_Y\circ f,s_Y\circ p_Y\circ f)=\ct_B^{*}(f)$ and since $D_B(id_Y,s_Y\circ p_Y)=\ct_B^{*}(Y)$, we conclude $\ct_B^{*}(f)\leq  \ct_B^{*}(Y)$. 
\end{proof}

\begin{corollary}\label{cor:phd-prod-inequality}
Let $f,g:X\to Y$ be fibrewise pointed maps. Then 
$$D_B(f,g)\leq (\ct_B^{*}(f)+1)\cdot (\ct_B^{*}(g)+1).$$
\end{corollary}

\begin{proof}
Recall that $\ct_B^*(f) = D_B(f, s_Y \circ p_X)$. Similarly, we have $\ct_B^*(g) = D_B(g, s_Y \circ p_X)$. Now, consider the open covers $\{U_0, \dots, U_n\}$ and $\{V_0, \dots, V_m\}$ of $X$, such that $f{|U_i} \simeq_B (s_Y \circ p_X)|_{U_i}$ for $0 \leq i \leq n$, and $g|_{V_j} \simeq_B (s_Y \circ p_X)|_{V_j}$ for $0 \leq j \leq m$, respectively.

We construct another open cover $\{W_{ij} = U_i \cap V_j \mid 0 \leq i \leq n, 0 \leq j \leq m\}$ of $X$, for which we have $f|_{W_{ij}} \simeq_B (s_Y \circ p_X)|_{W_{ij}} \simeq_B g|_{W_{ij}}$ for $0 \leq i \leq n$ and $0 \leq j \leq m$.
This construction establishes the desired inequality.
\end{proof}

\begin{proposition}\label{prop:phd-composition-ineq}
Let $h,h':Z\to X$ and $f,g:X\to Y$ be fibrewise maps such that $f\circ h'\simeq_B g\circ h'$. Then 
$D_B(f\circ h,g\circ h)\leq D_B(h,h').$    

\end{proposition}

\begin{proof}
Consider $U \subset Z$ an open subset such that $h|_{U} \simeq_B h'|_{U}$. Since $f \circ h' \simeq_B g \circ h'$, it follows that $f \circ h'|_{U} = (f \circ h')|_{U} \simeq_B (g \circ h')|_{U} = g \circ h'|_{U}$. Therefore, we have  
$$
(f \circ h)|_{U} = f \circ h|_{U} \simeq_B f \circ h'|_{U} \simeq_B g \circ h'|_{U} \simeq_B g \circ h|_{U} = (g \circ h)|_{U}.
$$  
By applying this argument to an open cover, the result follows.
\end{proof}

\begin{corollary}\label{prop:phd-ineq-ccatBdomain-TCBcodomain}
Let $f,g:X\to Y$ be two fibrewise pointed maps. Then \[D_B(f,g)\leq \mbox{min}\{\ct_B^{*}(X),\TC_B(Y)\}.\]   
\end{corollary}

\begin{proof}
Observe that $f \circ s_X \circ p_X = s_Y \circ p_X = g \circ s_X \circ p_X$. By setting $h = id_X$ and $h' = s_X \circ p_X$ in \eqref{prop:phd-composition-ineq}, we obtain the inequality  
$$
D_B(f, g) = D_B(f \circ id_X, g \circ id_X) \leq D_B(id_X, s_X \circ p_X) = \ct_B^*(X).
$$  
It follows from \Cref{prop:phd-equals-secatB} that $D_B(f, g) = \sct_B(\widetilde{\Pi}_Y)$. Since $\widetilde{\Pi}_Y$ is the fibrewise pullback of $\Pi_Y$, we deduce the inequality $\sct_B(\widetilde{\Pi}_Y) \leq \sct_B(\Pi_Y) = \TC_B(Y)$. 
\end{proof}

\begin{remark}
Note that the inequality $D_B(f, g) \leq \ct_B^{*}(X)$ can still hold even when $f$ and $g$ are not pointed, provided $X$ is considered a fibrewise pointed space and $Y$ is a vertically connected fibrewise space. By a vertically connected fibrewise space, we mean a fibrewise space where, up to fibrewise homotopy, there exists only one fibrewise map $B \to Y$. When $B$ is a point, this condition reduces to path-connectedness. When $B$ is a CW complex, a fibre bundle over $B$ with fibre $F$ is vertically connected if the dimension of $B$ does not exceed the connectivity of $F$.
\end{remark}

We can also recover a result proved in {\cite[Proposition 13]{Calcines-fibrewiseTC}}:

\begin{corollary}\label{cor:catBX-leq-TCBX-leq-catBXX}
 Let $X$ be a fibrewise pointed space. Then
 \[\ct_B^{*}(X)\leq \TC_B(X)\leq \ct_B^{*}(X\times_B X).\]    
\end{corollary}

\begin{proof}
Recall the fibrewise map $i_2:X\to X\times_B X$ defined by $i_2(x)=(s_X\circ p_X(x),x)$. Using part $(2)$ of \Cref{prop:composition-inequalities} and \Cref{cor:parametrizedTC-as-phd}, we obtain 
\[\ct_B^*(X)=D_B(s_X\circ p_X,id_X)=D_B(pr_1\circ i_2,pr_2\circ i_2)\leq D_B(pr_1, pr_2)=\TC_B(X).\]
The inequality $\TC_B(X)=D_B(pr_1,pr_2)\leq \ct_B^*(X\times_B X)$ follows from \Cref{prop:phd-ineq-ccatBdomain-TCBcodomain}.
\end{proof}

The following is a fibrewise homotopy invariance of the parametrized homotopic distance.

\begin{proposition}\label{prop:DBB-compositions-eq}
Let $f, g: X \to Y$ be fibrewise maps. Suppose there exists a fibrewise map $h: Y \to Y'$ with a left fibrewise homotopy inverse. Then, 
$
D_B(h \circ f, h \circ g) = D_B(f, g).
$

Similarly, if there is a fibrewise map $h: X' \to X$ with a right fibrewise homotopy inverse, then
$$
D_B(f \circ h, g \circ h) = D_B(f, g).
$$
  
\end{proposition}

\begin{proof}
Let $h': Y' \to Y$ be the left fibrewise homotopy inverse of $h$, meaning that $h' \circ h \simeq_B id_Y$. By applying \Cref{prop:composition-inequalities}, we obtain the following chain of inequalities:
$$
D_B(f, g) = D_B(h' \circ h \circ f, h' \circ h \circ g) \leq D_B(h \circ f, h \circ g) \leq D_B(f, g).
$$
This yields the desired equality. The proof for the second case follows analogously.
\end{proof}

As a consequence, the parametrized homotopic distance is a fibrewise homotopy invariant in the following sense: 

\begin{corollary}\label{invar} Given fibrewise maps $f, g: X \to Y$ and $f', g': X' \to Y'$, and fibrewise homotopy equivalences $\alpha: Y \to Y'$ and $\beta: X' \to X$ such that $\alpha \circ f \circ \beta \simeq_B f'$ and $\alpha \circ g \circ \beta \simeq_B g'$, we have
$D_B(f, g) = D_B(f', g').$
\end{corollary}

\begin{corollary}
If there is a fibrewise homotopy equivalence $f:X\rightarrow X'$, then $\TC_B(X)=\TC_B(X')$. Moreover, if $X$ and $X'$ are fibrewise pointed spaces, and $f$ is a fibrewise pointed map, then $\ct_B^*(X)=\ct_B^*(X')$.
\end{corollary}

Our goal is to prove that the parametrized homotopic distance satisfies the triangle inequality when the domain space is normal, thus defining a metric on the set of fibrewise homotopy classes $[X, Y]_B$. We will use the same tool as in the classical case of homotopic distance, studied by by E. Macías-Virgós and D. Mosquera-Lois.\cite{macias2022homotopic}.


\begin{proposition}\label{prop: triangle-inequality}
Let $f,g,h:X\rightarrow Y$ be fibrewise maps where $X$ is a normal space. Then
$$D_B(f,g)\leq D_B(f,h)+D_B(h,g).$$
\end{proposition}

\begin{proof}
Suppose that $D_B(f, g) = m$ and $D_B(g, h) = n$ and consider open covers $\{U_i\}_{i=0}^m$ and $\{V_j\}_{j=0}^m$ of $X$ such that $f|_{U_i} \simeq _B g|_{U_i}$ for all $i\in \{0, \dots, m \}$, and $g|_{V_j} \simeq _B h|_{V_j}$ for all $j\in \{0, \dots, n \}$. Then, by \cite[Lemma 4.3]{oprea2011mixing}, there exists a third open cover $\{W_k\}_{k=0}^{m+n}$ such that $f|_{W_k} \simeq _B g|_{W_k} \simeq _B h|_{W_k}$ for all $k$, thereby proving that $D_B(f, h) \leq m + n$.   
\end{proof}

\begin{remark}
Note that the above result also holds when $X$ is fibrewise homotopy equivalent to a fibrewise space $X'$ that is normal. It suffices to apply Corollary \ref{invar}.  \end{remark}

As a consequence to \Cref{prop: triangle-inequality}, we can improve the inequality in \Cref{cor:phd-prod-inequality}.

\begin{corollary}\label{cor:triangle-ineq}
For fibrewise pointed maps $f,g$ we have
$D_B(f,g)\leq \ct_B^*(f)+\ct_B^*(g).$
\end{corollary}
\begin{proof}
This follows by taking $h=s_Y\circ p_X$ in \Cref{prop: triangle-inequality}.
\end{proof}

Next, we present another result that follows from \cite[Lemma 4.3]{oprea2011mixing}.

\begin{proposition}\label{prop:triangle-ineq-compositions}
Consider fibrewise maps $f,g:X\rightarrow Y$ and $f',g':Y\rightarrow Z$ where $X$ is normal. Then
$D_B(f'\circ f,g'\circ g)\leq D_B(f,g)+D_B(f',g').$
\end{proposition}

\begin{proof}
Suppose that \( D_B(f, g) = m \) and \( D_B(f', g') = n \). Consider an open cover \( \{U_i\}_{i=0}^m \) of \( X \) and an open cover \( \{V_j\}_{j=0}^n \) of \( Y \), such that \( f|_{U_i} \simeq_B g|_{U_i} \) for all \( i \in \{0, \dots, m\} \), and \( f'|_{V_j} \simeq_B g'|_{V_j} \) for all \( j \in \{0, \dots, n\} \).
Clearly, for each \( i \in \{0, \dots, m\} \), we have \( (g' \circ f)|_{U_i} \simeq_B (g' \circ g)|_{U_i} \), and for each \( j \in \{0, \dots, n\} \), we have \( (f' \circ f)|_{V'_j} \simeq_B (g' \circ f)|_{V'_j} \), where \( V'_j := f^{-1}(V_j) \), and \( \{V'_j\}_{j=0}^n \) is also an open cover of \( X \).
By applying \cite[Lemma 4.3]{oprea2011mixing} again, we obtain a new open cover \( \{W_k\}_{k=0}^{m+n} \) of \( X \), such that \( (f' \circ f)|_{W_k} \simeq_B (g' \circ f)|_{W_k} \simeq_B (g' \circ g)|_{W_k} \) for each \( k \in \{0, \dots, m+n\} \).
\end{proof}

We also examine the behavior of parametrized homotopic distance under the product of fibrewise maps. Observe that given fibrewise maps \( f, g: X \to Y \) and \( h: X' \to Y' \), we can form the natural products \( f \times_B h, g \times_B h: X \times_B X' \to Y \times_B Y' \). It is straightforward to verify the inequality
$$ D_B(f \times_B h, g \times_B h) \leq D_B(f, g) $$
since \( f|_U \simeq_B g|_U \) implies \( (f \times_B h)|_{(U \times X') \cap (X \times_B X')} \simeq_B (g \times_B h)|_{(U \times X') \cap (X \times_B X')} \) for any subset \( U \subset X \). By symmetry, it is evident that we also have
$$ D_B(h \times_B f, h \times_B g) \leq D_B(f, g). $$

With additional conditions, equality can also be achieved. Specifically, suppose that \( X' \) is pointed. In this case, we can consider the fibrewise map \( i_1: X \to X \times_B X' \), defined by \( i_1(x) = (x, (s_{X'} \circ p_{X})(x)) \), and let \( pr_1: Y \times_B Y' \to Y \) be the projection onto the first factor. It is evident that
$ f = pr_1 \circ (f \times_B h) \circ i_1 \quad \text{and} \quad g = pr_1 \circ (g \times_B h) \circ i_1. $
Thus, by applying Proposition \ref{prop:composition-inequalities}, we obtain
\begin{equation}\label{eq:fibrewise-prod}
D_B(f, g) = D_B(pr_1 \circ (f \times_B h) \circ i_1, pr_1 \circ (g \times_B h) \circ i_1) \leq D_B(f \times_B h, g \times_B h).    
\end{equation}

Now we are prepared for our next result:

\begin{proposition}\label{prop:prod-inequality}
Consider the fibrewise maps $f, g: X \to Y$ and $f', g': X' \to Y'$. If the space $X \times_B X'$ is normal, then 
$D_B(f \times_B f', g \times_B g') \leq D_B(f, g) + D_B(f', g').$
\end{proposition}

\begin{proof}
Applying Proposition \ref{prop: triangle-inequality} along with the preceding comments, we obtain
$$D_B(f\times_Bf',g\times_Bg')\leq D_B(f\times _Bf',g\times _Bf')+D_B(g\times _Bf',g\times _Bg')\leq D_B(f,g)+D_B(f',g').\qedhere$$
\end{proof}

\begin{remark}
It is worth noting that the condition that $X \times_B X'$ is normal is not too restrictive. For example, if both $X$ and $X'$ are metrizable and $B$ is Hausdorff, then $X \times_B X'$ is normal.
\end{remark}

\begin{corollary}\label{cor:prod-ineq-catB}
Let $X$ and $X'$ be fibrewise pointed spaces, both of which are metrizable, and assume that $B$ is Hausdorff. Then, $
\ct_B^*(X \times_B X') \leq \ct_B^*(X) + \ct_B^*(X').$
\end{corollary}

\begin{proof}
Consider $f = id_X$, $f' = id_{X'}$, $g = s_X \circ p_X$, and $g' = s_{X'} \circ p_{X'}$, and apply \Cref{prop:prod-inequality}. Note that $B \times_B B = B$.
\end{proof}

As an application of \Cref{prop:prod-inequality}, we generalize \cite[Corollary 6.2]{C-F-W}.

\begin{corollary}\label{cor:tcB-prod-inequality}
Let $X$ and $X'$ be fibrewise spaces such that both $X$ and $X'$ are metrizable, and assume that $B$ is Hausdorff. Then,
$\TC_B(X \times_B X') \leq \TC_B(X) + \TC_B(X').
$
\end{corollary}

\begin{proof}
Consider $f = pr_1: X \times_B X \to X$, $f' = pr_1: X' \times_B X' \to X'$, $g = pr_2: X \times_B X \to X$, and $g' = pr_2: X' \times_B X' \to X'$. We now apply \Cref{prop:prod-inequality}. 
Note that there is an obvious fibrewise homeomorphism that makes the following diagram commute:
\[
\xymatrix{
{(X \times_B X') \times_B (X \times_B X')} \ar[rr]^{\cong} \ar[dr]_{pr_{\varepsilon}} & & {(X \times_B X) \times_B (X' \times_B X')} \ar[dl]^{pr_{\varepsilon} \times_B pr_{\varepsilon}} \\
 & {X \times_B X'} & 
}
\]
for all $\varepsilon \in \{1,2\}$.
\end{proof}

\section{Fibrewise $H$-spaces}\label{sec: fibrewise-H-spaces}
In this section, we study fibrewise $H$-spaces and their relation to fibrewise homotopic distance. A \emph{fibrewise $H$-space} is a fibrewise pointed space $X$ equipped with a fibrewise pointed map
$\mu : X \times_B X \to X, \hspace{5pt} (x, y) \mapsto \mu(x, y) = x \cdot y,$
called \emph{fibrewise multiplication}, such that $\mu \circ i_1\simeq ^B_B id_X$ and $\mu \circ i_2\simeq ^B_Bid_X$,
where $i_1$ and $i_2$ denote the natural inclusions of $X$ into $X \times_B X$; that is, $i_1 = (id_X, s_X \circ p_X)$ and $i_2 = (s_X \circ p_X, id_X)$. The fibrewise space $X$ is \emph{homotopy associative} if, in addition, we have a fibrewise pointed homotopy $\mu \circ (\mu \times _Bid_X)\simeq ^B_B \mu \circ (id_X\times _B\mu ).$ 

A \emph{fibrewise homotopy right inverse} for a fibrewise multiplication $ \mu $ on $ X $ is a fibrewise pointed map $ u: X \to X $ such that $ \mu \circ (id_X, u) $ is fibrewise pointed nullhomotopic, meaning it is fibrewise pointed homotopic to $ s_X \circ p_X $. A similar definition holds for fibrewise homotopy left inverses.

When $ \mu $ is fibrewise homotopy associative, a fibrewise pointed map is a homotopy right inverse if and only if it is also a fibrewise homotopy left inverse. In this case, we simply refer to it as a fibrewise homotopy inverse. The standard notation for this is $ u(x) = x^{-1} $, for all $ x \in X $.

A \emph{fibrewise group-like space} is a homotopy associative fibrewise $H$-space $ X $ where the fibrewise multiplication admits a fibrewise homotopy inverse.

For further details on fibrewise $H$-spaces, see \cite{C-J}.

\begin{remark}\label{group-like}
Using reasoning similar to the classical case (see, for example, \cite[page 119]{whitehead2012elements}), it can be easily verified that if $X$ is a homotopy associative fibrewise $H$-space, then $X$ is a fibrewise group-like space if and only if the fibrewise shearing map
$$sh:X\times _BX\rightarrow X\times _BX,\hspace{5pt}(x,y)\mapsto (x,x\cdot y)$$ \noindent is a fibrewise pointed homotopy equivalence.
\end{remark}

We are particularly interested in fibrewise $H$-spaces that admit a fibrewise division:

\begin{definition}
Let $X$ be a fibrewise $H$-space. A \emph{fibrewise division} on $X$ is a fibrewise pointed map $\delta :X\times _BX\rightarrow X$ such that there is a fibrewise pointed homotopy $\mu \circ (pr_1,\delta )\simeq ^B_B pr_2,$ \noindent where $pr_1,pr_2:X\times _BX\rightarrow X$ are the respective projection maps.
\end{definition}

We now examine a related condition for the existence of a fibrewise division map.

\begin{lemma}\label{division}
Let $X$ be a fibrewise $H$-space. Then $X$ admits a fibrewise division if and only if the fibrewise shearing map
$sh=(pr_1,\mu )$ has a fibrewise pointed homotopy right inverse.
\end{lemma}

\begin{proof}
Suppose $\delta :X\times _BX\rightarrow X$ a fibrewise division, and define $\psi :=(pr_1,\delta )$. Then we have
$sh\circ \psi =(pr_1,\mu )\circ (pr_1,\delta )=(pr_1, \mu \circ (pr_1, \delta ))\simeq ^B_B (pr_1,pr_2)=id.$

Conversely, let $\psi: X \times_B X \to X \times_B X$ be a fibrewise pointed homotopy right inverse of $sh$, and consider its components $\psi = (\psi_1, \psi_2)$, where, $\psi_1, \psi_2: X \times_B X \to X$. Define $\delta := \psi_2$, which will act as a division map. Indeed, observe the following:
$$(pr_1, pr_2) = id\simeq^B_B sh \circ \psi = (pr_1, \mu) \circ (\psi_1, \psi_2)= (pr_1, \mu) \circ (\psi_1, \delta)= (\psi_1, \mu \circ (\psi_1, \delta)).$$
Thus, by separating components, we obtain $pr_1 \simeq^B_B \psi_1$ and $pr_2\simeq^B_B \mu \circ (\psi_1, \delta)\simeq^B_B \mu \circ (pr_1, \delta).$
\end{proof}

A notable example of $H$-spaces that admit a fibrewise division is provided by fibrewise group-like spaces.

\begin{example}
If $X$ is a fibrewise group-like space, then by Remark \ref{group-like} the fibrewise shearing map is a fibrewise pointed homotopy equivalence. As a result, by Lemma \ref{division}, $X$ admits a fibrewise division map $\delta: X \times_B X \to X$ given by $\delta(x, y) = x^{-1} \cdot y$.
\end{example}

However, a fibrewise $H$-space does not necessarily need to be group-like in order to admit a fibrewise division map. In the following, we explore sufficient conditions for the existence of such a map. To establish these conditions, we first present a preliminary lemma, which will play a crucial role in our argument. The proof of this lemma can be found in \cite{dold1963partitions}.

Recall that a space $B$ is \emph{numerably categorical} if it has a numerable cover $\mathcal{U} = \{V_j\}_{j \in J}$ such that each inclusion map $V_j \hookrightarrow B$ is nullhomotopic for all $j \in J$. This class of spaces is quite broad. For instance, CW-complexes, locally contractible paracompact spaces (such as topological manifolds), and classifying spaces $B_G$ all belong to this class. Moreover, by a \emph{weak fibration} (or Dold fibration) we will mean a map $p: E \rightarrow B$ that is fibrewise homotopy equivalent to a Hurewicz fibration (as fibrewise spaces over $B$). Equivalently, $p$ is a weak fibration if it satisfies the weak covering homotopy property (WCHP) for any space $X$, as described in \cite[Definition 5.1]{dold1963partitions}. 
A fibrewise space $X$ is said to be \emph{weakly fibrant} if the projection map $p_X: X \rightarrow B$ is a weak fibration. 

With these definitions in place, we now turn to the aforementioned lemma, which will provide the necessary framework for understanding the sufficient conditions for the existence of a fibrewise division map.

\begin{lemma}\cite[Theorem 6.3]{dold1963partitions}\label{dold}
Let $f:X\rightarrow Y$ be a fibrewise map over a numerably categorical space $B$, where both $X$ and $Y$ are weakly fibrant. The following statements are equivalent:
\begin{enumerate}
\item $f$ is a fibrewise homotopy equivalence.

\item The restriction of $f$ to every fibre, denoted by $f_b:p_X^{-1}(\{b\})\rightarrow p_Y^{-1}(\{b\})$ for each $b\in B,$ is an ordinary homotopy equivalence.
\end{enumerate}
\end{lemma}

\begin{remark}\label{virem}
In fact, as demonstrated by A. Dold, the result remains valid even if statement (2) is replaced with the condition that $f_b$ is a homotopy equivalence for at least one point $b \in B$ in each path component of $B$. Specifically, if $B$ is path-connected, the result holds as long as there exists a single point $b \in B$ such that $f_b$ is a homotopy equivalence.
\end{remark}

Now, we are ready to state and prove our result:

\begin{proposition}\label{prop: existence-division-map}
Let $X$ be a fibrewise well-pointed $H$-space over a path-connected, numerably categorical space $B$. Suppose that $X$ is weakly fibrant, and there exists a point $b \in B$ such that the fibre $X_b=p_X^{-1}(\{b\})$ is a connected CW-complex. Then, the fibrewise shearing map $sh: X \times_B X \rightarrow X \times_BX$ is a fibrewise pointed homotopy equivalence. Consequently, $X$ admits a fibrewise division map $\delta: X \times_B X \rightarrow X$.
\end{proposition}

\begin{proof}
Since $X$ is a fibrewise $H$-space, then the restriction of the multiplication map $\mu$ to each fibre, $\mu_b : X_b \times X_b \to X_b$, equips the fibre $X_b$ with the structure of an ordinary $H$-space. Consequently, as $X_b$ is a connected CW $H$-space, the ordinary shearing map, $sh_b : X_b \times X_b \to X_b \times X_b$ (specifically, the restriction of $sh$ to the fibres),
is a (pointed) homotopy equivalence. Since $X \times_B X$ is weakly fibrant, Lemma \ref{dold} (and/or Remark \ref{virem}) implies that the fibrewise shearing map $sh : X \times_B X \to X \times_B X$
is a fibrewise homotopy equivalence. Furthermore, since $X \times_B X$ is fibrewise well-pointed by Proposition \ref{important2}, and since $p_X:X\rightarrow B$ is always a fibrewise fibration, it follows from Proposition \ref{important1} (iii) that $sh$ is a fibrewise pointed homotopy equivalence.
\end{proof}

\begin{remark}
If $B$ is not path-connected, the result still holds by requiring that for each path component of $B$, there exists a point $b \in B$ such that the fibre $X_b$ is a connected CW-complex.
\end{remark}

As an application of \Cref{prop: existence-division-map}, we now present an example of a fibrewise $H$-space that admits a division map but is not a fibrewise group-like space.

For a compact Lie group $G$, Cook and Crabb construct fibrewise Hopf structures on sphere bundles using principal $G$-bundles and $G$-equivariant Hopf structures on spheres in \cite{cook1993fibrewise}. To construct our desired example, we briefly recall their general construction: Let $B$ be a finite CW-complex, and let $q: P \to B$ be a principal $G$-bundle over $B$. Consider an odd-dimensional vector space $V$ on which $G$ acts orthogonally. This action induces a vector bundle $\xi := P \times_G V$ over $B$. Now, consider the vector bundle $\mathbb{R} \oplus \xi$ over $B$, and the corresponding sphere bundle, denoted by $S(\mathbb{R} \oplus \xi) = P \times_G S(\mathbb{R} \oplus V)$. This construction provides a fibrewise Hopf structure on $S(\mathbb{R} \oplus \xi)$, extending the equivariant Hopf structure of the fibre $S(\mathbb{R} \oplus V)$.
Now, we are ready to describe our desired example.

\begin{example}\label{ex:fibHspace} Let $G_2$ denote the automorphism group of the octonions, which is a subgroup of the orthogonal group $O(7)$. As shown in \cite[Proposition 2.1]{cook1993fibrewise}, the Hopf structure on $S(\mathbb{R} \oplus \mathbb{R}^7) = S^7$ given by Cayley multiplication is $G_2$-equivariant. Moreover, it is well known that this Hopf structure is non-homotopy associative and admits a division map.
Note that $G_2$ acts orthogonally on $\mathbb{R}^7$. Suppose $P \to B$ is a principal $G_2$-bundle as in the previous general discussion, and $\xi = P \times_{G_2} \mathbb{R}^7$. Then, we have a fibrewise Hopf structure on the sphere bundle $S(\mathbb{R} \oplus \xi)$, which extends the $G_2$-equivariant Hopf structure of $S^7$. This fibrewise Hopf structure is clearly not fibrewise homotopy associative. On the other hand, it follows from \Cref{prop: existence-division-map} that the fibrewise $H$-space $S(\mathbb{R} \oplus \xi)$ admits a fibrewise division map. Thus, this provides an example of a fibrewise $H$-space with a division map that is not fibrewise group-like.

As an additional noteworthy observation, estimating the parametrized topological complexity for this particular class of fibrewise $H$-spaces is relatively straightforward. Specifically, by applying the homotopy dimension-connectivity upper bound from \Cref{thm:lower-upper-bounds} (or alternatively using \cite[Proposition 7.2]{C-F-W}, since $S(\mathbb{R} \oplus \xi)$ is fibrant), we obtain: 
$
\TC_B(S(\mathbb{R} \oplus \xi)) \leq 2 + \frac{\dim(B)}{7}.
$

Moreover, since the fibre of $S(\mathbb{R} \oplus \xi)$ is $S^7$, which is not contractible, the following inequality can be established using \cite[Proposition 4.5]{C-F-W}:  
$$
1 \leq \TC_B(S(\mathbb{R} \oplus \xi))  \leq 2 + \frac{\dim(B)}{7}.
$$  
In particular, if $\dim(B) \leq 6$, then $\TC_B(S(\mathbb{R} \oplus \xi))$ is either $1$ or $2$.

\end{example}

We now present an interesting result that connects the parametrized homotopic distance to the fibrewise unpointed Lusternik-Schnirelman (LS) category, specifically when considering fibrewise $H$-spaces with a fibrewise division.

\begin{theorem}\label{thm:phd-leq-pcat}
Let $X$ be a fibrewise $H$-space with a fibrewise division map, and let $f,g:X\times_B X\rightarrow X$ be fibrewise pointed maps. Then, we have the inequality $D_B(f,g)\leq \ct_B^*(X).$
\end{theorem}

\begin{proof}  
Let $U$ be a fibrewise categorical subset of $G$, meaning that $i|_{U} \simeq_B s_X \circ p_{X}|_U$. Consider the division map $\delta: X \times_B X \to X$, and the composite map
$\phi := \delta \circ (f, g): X \times_B X \to G,$
along with the open set $V := \phi^{-1}(U)$. We have the following strictly commutative diagram:
$$
\xymatrix{
{V} \ar[d]_{\phi|_V} \ar@{^{(}->}[rr]^{\mathrm{inc}_V} & & {X \times_B X} \ar[d]^{\phi} \\
{U} \ar@{^{(}->}[rr]_{\mathrm{inc}_U} & & {X}.
}$$

This gives rise to the following sequence:
$$
\begin{array}{lcl}
g|_V &=& g \circ \mathrm{inc}_V \\
     &=& pr_2 \circ (f, g) \circ \mathrm{inc}_V \\
     &\simeq^B_B& \mu \circ (pr_1, \delta) \circ (f, g) \circ \mathrm{inc}_V \\
     &=& \mu \circ (f, \delta \circ (f, g)) \circ \mathrm{inc}_V \\
     &=& \mu \circ (f, \phi) \circ \mathrm{inc}_V \\
     &=& \mu \circ (f \circ \mathrm{inc}_V, \phi \circ \mathrm{inc}_V) \\
     &=& \mu \circ (f|_V, \mathrm{inc}_U \circ \phi|_V) \\
     &\simeq_B& \mu \circ (f|_V, s_X \circ p_X|_U \circ \phi|_V) \\
     &=& \mu \circ (f|_V, s_X \circ p_X \circ \phi|_V) \\
     &=& \mu \circ (f|_V, s_X \circ p_X \circ pr_1|_V) \\
     &=& \mu \circ (f|_V, s_X \circ p_X \circ f|_V) \\
     &=& \mu \circ (id_X, s_X \circ p_X) \circ f|_V \\
     &\simeq^B_B& id_X \circ f|_V = f|_V.
\end{array}
$$
Thus, we have $g|_{V} \simeq_B f|_{V}$. By applying this argument to open covers, we can derive the desired inequality.
\end{proof}

Lupton and Scherer \cite[Theorem 1]{L-S} showed that the topological complexity of a path-connected CW $H$-space equals its LS category. The following result provides a fibrewise analogue.

\begin{corollary}\label{cor:tcB-equals-catB}
Let $X$ be a fibrewise $H$-space with a fibrewise division map. Then
$\TC_B(X) = \ct_B^{*}(X).$
\end{corollary}

\begin{proof}
The inequality $\TC_B(X) \leq \ct_B^{*}(X)$ follows by setting $f =pr_1$ and $g =pr_2$ in \Cref{thm:phd-leq-pcat}. The reverse inequality is a consequence of \Cref{cor:catBX-leq-TCBX-leq-catBXX}.
\end{proof}

In the following result, we use the product $f\cdot g:=\mu \circ (f,g)=\mu \circ (f\times _Bg)\circ \Delta $ of two fibrewise maps $f,g:X\rightarrow Y$ when $Y$ is a fibrewise $H$-space.
Here, $\Delta : X \to X \times_B X$ is the diagonal map, $f \times_B g : X \times_B X \to Y \times_B Y$ is the fibrewise product map, and $\mu : Y \times_B Y \to Y$ is the fibrewise multiplication map in the fibrewise $H$-space structure on $Y$. 

\begin{proposition}\label{prop:multiplication-fib-spaces-ineq}
Let $f,g,h:X\rightarrow Y$ be fibrewise maps where $Y$ is a fibrewise $H$-space. Then
$$D_B(f\cdot h,g\cdot h)\leq D_B(f,g).$$
\end{proposition}

\begin{proof}
We already know that $D_B(f\times _Bh,g\times _Bh)\leq D_B(f,g)$. Consequently, by applying Proposition \ref{prop:composition-inequalities}, we obtain
$$D_B(f\cdot h,g\cdot h)=D_B(\mu \circ (f\times _Bh)\circ \Delta ,\mu \circ (g\times _Bh)\circ \Delta )\leq D_B(f\times _Bh,g\times _Bh)\leq D_B(f,g).$$
\end{proof}

\begin{corollary}\label{cor:grplike-space}
If $X$ is a fibrewise group-like space, then $D_B(\mu ,\delta )=D_B(id_X,u),$ \noindent where $\delta :X\times _BX\rightarrow X$ is the fibrewise division map, and $u:X\rightarrow X$ represents the fibrewise homotopy inverse.
\end{corollary}

\begin{proof}
Recall the fibrewise pointed map $i_1=(id_X,s_X\circ p_X):X\rightarrow X\times _BX.$ Since $\mu \circ i_1\simeq ^B_B id_X$ and $\delta \circ i_1\simeq ^B_Bu$, we obtain
$D_B(id_X,u)=D_B(\mu \circ i_1,\delta \circ i_1)\leq D_B(\mu ,\delta ).$
Conversely, since $\mu =pr_1\cdot pr_2=(id_X\circ pr_1)\cdot pr_2$ and $\delta =(u\circ pr_1)\cdot pr_2$, we have \[D_B(\mu ,\delta )\leq D_B(id_X\circ pr_1,u\circ pr_1)\leq D_B(id_X,u).\qedhere\]
\end{proof}

Now, assuming that a fibrewise space $X\times_B X$ is normal, we generalize the inequality stated in \Cref{prop:multiplication-fib-spaces-ineq}.

\begin{proposition}\label{prop:fib-mult-ineq}
Let $f,g,h,h':X\rightarrow Y$ be fibrewise maps, where $Y$ is a fibrewise $H$-space and $X\times_B X$ is normal. Then
$D_B(f\cdot h,g\cdot h')\leq D_B(f,g) + D_B(h,h').$
\end{proposition}

\begin{proof}
Using the triangle inequality, the fact $D_B(f\times _Bh,g\times _Bh)\leq D_B(f,g)$ and \Cref{prop:composition-inequalities}, we obtain
$$\begin{array}{ll}
D_B(f\cdot h,g\cdot h')& = D_B(\mu \circ (f\times _Bh)\circ \Delta ,\mu \circ (g\times _Bh')\circ \Delta )\\
& \leq D_B(f\times _Bh,g\times _Bh')\\
&\leq D_B(f\times _Bh, g\times _Bh)+ D_B(g\times _Bh, g\times _Bh')\\
&\leq D_B(f,g)+ D_B(h,h').\qedhere
\end{array}
$$
\end{proof}

\section{Fibrewise fibrations}\label{sec: fibrewise fibrations}

In this section, we aim to estimate the parametrized homotopic distance of fibre-preserving maps between fibrewise fibrations. Our approach relates this distance to the parametrized homotopic distance of the induced maps on individual fibres and the fibrewise unpointed LS category of the base space.

We begin by recalling the concept of a fibre in the fibrewise setting.
Let $X$ be a fibrewise pointed space over $B$ with section $s_X:B\to X$, and let $\pi:E\to X$ be a fibrewise fibration. 
The fibre of $\pi:E\to X$ at $s_X$ is given by the following pullback:
$$\xymatrix{
{F} \ar[r]^i \ar[d] & {E} \ar[d]^{\pi } \\
{B} \ar[r]_{s_X} & {X.}
}$$
Since $s_X$ is an embedding, we can equivalently consider $F=\pi^{-1}(s_X(B))$, with $i:F\hookrightarrow E$ being the natural inclusion, and the obvious projection $F\stackrel{\pi |_F}{\longrightarrow }s_X(B)\stackrel{s_X^{-1}}\longrightarrow B.$

Let $X$, $X'$ be fibrewise pointed spaces over $B$. Suppose $\pi:E\to X$ and $\pi':E'\to X'$ are fibrewise fibrations with fibres $F$ and $F'$, respectively. Let $f,g:E\rightarrow E'$ be fibrewise maps, and $\bar{f}$, $\bar{g}$ be fibrewise pointed maps satisfying  $\pi'\circ f=\bar{f}\circ \pi$ and $\pi'\circ g=\bar{g}\circ \pi$:
\begin{equation}\label{dia:comm-diagram-fibrewise-fibration}
\xymatrix{
{E} \ar[d]_{\pi } \ar[rr]^{f,g} & & {E'} \ar[d]
^{\pi '} \\
{X} \ar[rr]_{\bar{f},\bar{g}} & & {X'.}
}
\end{equation}

\noindent Since $\bar{f}\circ s_X=s_{X'}=\bar{g}\circ s_X$ it follows $f(F)\subset F'$ and $g(F)\subset F'$. Therefore, we obtain induced fibrewise maps $f_0=f|_{F}:F\to F'$ and $g_0=g|_{F}:F\to F'$ between the corresponding fibres. Under these hypotheses, we have the following theorem:

\begin{theorem}\label{thm: fibrewise-fibration}
$D_B(f,g)+1\leq (D_B(f_0,g_0)+1)\cdot (\ct_B^{*}(X)+1).$
\end{theorem}

\begin{proof}
Suppose $\ct_B^*(X)=n$ and $D_B(f_0,g_0)=m$.
Consider $\{U_0,\dots,U_n\}$ a fibrewise categorical open cover of $X$, and let $\{V_0,\dots,V_m\}$ be an open cover of $F_0$ such that $f_{0}|_{Vj}\simeq_B g_{0}|_{Vj}$ for $0\leq j\leq m$. For each $0\leq j\leq m$, denote this fibrewise homotopy by $F_j:I_B(V_j)\to F'$.

Since each $U_i$ is a fibrewise categorical open subset of $X$, there exists a fibrewise homotopy $H_i:I_B(U_i)\to X$ with $H_i:\mathrm{inc}_{U_i}\simeq_B s_X\circ p_{X}|_{U_i}$. Define $U'_i:=p^{-1}(U_i)$. Then, applying the fibrewise homotopy lifting property for $\pi$, we obtain a lifted homotopy $\tilde{H}_i:I_B(U'_i)\to E$ that makes the following diagram commute:
$$\xymatrix{
{U'_i} \ar[d]_{i_0} \ar@{^{(}->}@<-2pt>[rrr] & & &  {E} \ar[d]^{\pi } \\
{I_B(U'_i)} \ar@{.>}[urrr]^{\tilde{H}_i} \ar[rr]_{I_B(\pi )} & & {I_B(U_i)} \ar[r]_(.6){H_i} & {X}.
}$$
Since $\pi\circ \tilde{H}_i=H_i\circ I_B(\pi )$, we have $\pi (\tilde{H}_i(x,1))=H_i(\pi(x),1)=(s_X\circ p_X)(\pi(x))\in s_X(B).$ In other words,
    $\tilde{H}_i(x,1)\in \pi^{-1}(s_X(B))=F$.
This defines a fibrewise map $\tilde{H}_{i,1}:=\tilde{H}_i(-,1):U_i'\to F$.
Now define $W_{i,j}=U_i'\cap V_j'$ for $i\in \{0,1,\dots, n\}$ and $j\in \{0,1,\dots, m\}$, where $V_j'=\tilde{H}_{i,1}^{-1}(V_j)$. 
Finally, to show that \( f|_{W_{i,j}} \simeq_B g|_{W_{i,j}} \) for all \( i \in \{0,1,\dots, n\} \) and \( j \in \{0,1,\dots, m\} \), we use the same homotopy defined in the non-fibrewise setting, as shown in the proof of \cite[Theorem 6.1]{macias2022homotopic}.
\end{proof}

Let $p: E \to X$ be a fibration with fibre $F$ in the non-fibrewise setting. Varadarajan \cite{V} established the inequality  
\begin{equation}\label{eq:Varadarajan-cat-ineq}
\ct(E) + 1 \leq (\ct(F) + 1) \cdot (\ct(X) + 1).
\end{equation}
We now prove its fibrewise analogue. Observe that when $B$ is a point, this reduces to \eqref{eq:Varadarajan-cat-ineq}.

\begin{corollary}\label{cor:cib-fibration1}
Let $E$ and $X$ be fibrewise pointed spaces over $B$, and let $\pi:E\to X$ be a fibrewise fibration that is also fibrewise pointed. Then, the following inequality holds: 
$$ \ct_{B}^*(E)+1\leq (\ct_B^*( F)+1) \cdot(\ct_B^*(X)+1).$$
\end{corollary}

\begin{proof}
By setting $E'=E$, $X'=X$, $f=id_E$, $g=s_E\circ p_E$, $\bar{f}=id_X$, $\bar{g}=s_X\circ p_X$ in \eqref{dia:comm-diagram-fibrewise-fibration}, we obtain the following commutative diagram:
$$\xymatrix{
{E} \ar[rr]^{id_E,~ s_E\circ p_E} \ar[d]_{\pi } & & {E} \ar[d]^{\pi} \\
{X} \ar[rr]_{id_X,~ s_X\circ p_X} & & {X.}
}$$
Note that $f_0=id_F$ and $g_0=s_F\circ p_F$, where $s_F$ and $p_F$ are the natural restrictions of $s_X$ and $p_X$ to the fibre $F$.
Therefore, by applying \Cref{thm: fibrewise-fibration} to the diagram above, we obtain:
$$\ct_{B}^*(id_E, s_E\circ p_E)+1\leq (\ct_B^*( id_F, s_F\circ p_F)+1)\cdot (\ct_B^*(id_X, s_X\circ p_X)+1) .\qedhere$$
\end{proof}

For a fibration $F\hookrightarrow E\to X$, Farber and Grant \cite{F-G} proved the inequality
\begin{equation}\label{eq:Farber-Grant-inequality}
    \TC(E)+1\leq (\TC(F)+1)\cdot (\ct(X\times X)+1).
\end{equation} 
The following corollary provides its fibrewise analogue. Observe that when $B$ is a point, we recover \eqref{eq:Farber-Grant-inequality}.

\begin{corollary}\label{cor:cib-fibration2}
Let $E$ and $X$ be fibrewise pointed spaces over $B$ and let $\pi:E\to X$ be a fibrewise fibration. Then
$\TC_B(E)+1\leq (\TC_B(F)+1)\cdot (\ct_B^*(X\times_B X)+1).$ 
\end{corollary}

\begin{proof}
Consider the following commutative diagram as in \Cref{dia:comm-diagram-fibrewise-fibration}. 
$$\xymatrix{
{E\times _BE} \ar[d]_{\pi \times _B\pi} \ar[rr]^{pr_1,~ pr_2} & & {E} \ar[d]^{\pi } \\
{X\times _BX} \ar[rr]_{pr_1,~ pr_2} & & {X.}
}$$
Observe that $f_0=pr_1:F\times_B F\to F$ and $g_0=pr_2:F\times_B F\to F$.
Thus, the desired inequality follows again from \Cref{thm: fibrewise-fibration}.
\end{proof}

\section{Parametrized pointed homotopic distance}\label{sec:pphd}
In this section, we introduce the pointed version of the parametrized homotopic distance, defined within the context of fibrewise pointed homotopy. We show that its value closely matches the non-pointed version and, under mild conditions on dimension and connectivity, both invariants coincide. We begin by defining the parametrized pointed homotopic distance.

\begin{definition}\label{def:pphd}
Let $f,g:X\to Y$ be two fibrewise pointed maps between fibrewise pointed spaces $X$ and $Y$ over $B$. The parametrized pointed homotopic distance between $f$ and $g$, denoted by $D_B^B(f,g)$, is defined as the smallest positive integer $n$ for which there exists an open cover $\{U_0,\dots, U_n\}$ of $X$ with $s_X(B)\subset U_i$ and  $f|_{U_i}\simeq_{B}^Bg|_{U_i}$ for $0\leq i\leq n$.  If no such open cover exists, we set $D_B^B(f,g)=\infty$. 
\end{definition}

The following statements follow directly from \Cref{def:pphd}:
\begin{enumerate}
    \item $D_B^B(f,g)=D_B^B(g,f)$
    \item $D_B^B(f,g)=0$ if, and only if, $f\simeq_B^B g$.
    \item  If $f\simeq_B^B f'$ and $g\simeq_B^B g'$, then  $D_B^B(f,g)=D_B^B(f',g')$.
\end{enumerate}

The numerical invariant associated with the parametrized pointed homotopic distance is the so-called \emph{fibrewise pointed sectional category}, which was introduced in \cite{GC}.

\begin{definition}
The \emph{fibrewise pointed sectional category} of a fibrewise pointed map $f: E \rightarrow X$, denoted by $\mbox{secat}_B^B(f)$, is the smallest integer $n \geq 0$ such that $X$ admits an open cover $\{U_i\}_{i=0}^n$ with $s_X(B) \subset U_i$ and, for each $i$, there exists a fibrewise pointed homotopy section $s_i: U_i \rightarrow E$ of $f$, meaning that $f \circ s_i \simeq_B^B \mbox{inc}_{U_i}$, for each $i$. 
If no such $n$ exists, we define $\mbox{secat}_B^B(f) = \infty$.
\end{definition}

Similarly to the non-pointed case, when $f:E\rightarrow X$ is a fibrewise pointed fibration, the definition of fibrewise pointed sectional category can be refined by requiring the triangles in the diagram to be strictly commutative.

For fibrewise pointed maps $f,g:X\rightarrow Y,$ consider the pullback diagram:
\begin{equation}\label{eq:ptd-pullback-pathfib}
\xymatrix{
{\mathcal{P}_B(f,g)} \ar[d]_{\widetilde{\Pi} _Y} \ar[r] & {P_B(Y)} \ar[d]^{\Pi _Y} \\
{X} \ar[r]_{(f ,g )} & {Y\times _BY.} }    
\end{equation}
Note that $\widetilde{\Pi}_Y$ is a fibrewise pointed fibration, as $\Pi_Y$ is, and they are connected by this pullback. Similar to the non-pointed case, the parametrized pointed homotopic distance can be expressed in terms of the fibrewise pointed sectional category. The proof of this result is analogous and is therefore omitted.

\begin{proposition}\label{prop:DBB-is-secatB}
$D_B^B(f,g)=\mbox{secat}_B^B(\widetilde{\Pi }_Y)$. 
\end{proposition}

Other notable examples of numerical invariants expressible in terms of the fibrewise pointed sectional category include the fibrewise pointed LS category, $\mbox{cat}^B_B(X)$, and the fibrewise pointed topological complexity, $\mbox{TC}^B_B(X)$, both studied in \cite{Calcines-fibrewiseTC} and \cite{GC}. Specifically, for any fibrewise pointed space $X$, we have $\mbox{cat}^B_B(X)=\mbox{secat}^B_B(s_X)$ and $\mbox{TC}^B_B(X)=\mbox{secat}^B_B(\Pi_X)$.
Applying similar ideas as in the proofs of \Cref{prop:fLScat=DBi1i2} and \Cref{cor:parametrizedTC-as-phd}, we can establish the following result:

\begin{proposition}
Let $X$ be a fibrewise pointed space. Then
$$\ct_B^B(X)=D_B^B(i_1,i_2)\hspace{8pt}\mbox{and}\hspace{8pt}\TC_B^B(X)=D_B^B(pr_1,pr_2).$$  
\end{proposition}

Having analyzed the properties of the parametrized homotopic distance in \Cref{sec:phd-properties}, we may proceed analogously for its pointed version. Key properties include its behavior under compositions and their consequences, fibrewise pointed homotopy invariance, the triangle inequality for the parametrized pointed homotopic distance, and its behavior under the product of maps. Additional results concern fibrewise $H$-spaces, their interaction with the multiplication of fibrewise pointed maps, and fibrewise pointed fibrations.
Since the statements and proofs follow the same structure, with only minor adjustments (mutatis mutandis), we omit them. We leave the details to the reader, presenting them as an instructive exercise.

\subsection{Cohomological lower bound}
To study cohomological bounds on the fibrewise LS category, James and Morris \cite{J-M} introduce the fibrewise cohomology ring, defined as the quotient
$$H^*(X) / \langle p_X^*(H^*(B)) \rangle,$$
where $\langle p_X^*(H^*(B)) \rangle$ denotes the ideal generated by the image of $p_X^*(H^*(B))$. However, when $X$ is a fibrewise pointed space, a more manageable version of this ring can be considered.
Here and throughout, we consider singular cohomology with coefficients in a commutative ring $R$ with unity, omitting $R$ from the notation.

We first observe that from the split short exact sequence of graded abelian groups
$$
\xymatrix{
{0} \ar[r] & {H^*(B)} \ar[r]^{p_X^*} & {H^*(X)} \ar@/^1pc/[l]^{s_X^*} \ar[r] & {H^*(X) / \langle p_X^*(H^*(B)) \rangle} \ar[r] & {0}
}
$$
\noindent we obtain the isomorphism
$H^*(X)/\left<p_X^*(H^*(B))\right>\cong \mathrm{ker}(s_X^*)$.
Moreover, if  $H^*(X,s_X(B))$ denotes the cohomology of the pair $(X,s_X(B))$, then there is an isomorphism:
$$
H^*(X,s_X(B))\cong \mathrm{ker}(s_X^*).
$$    
Indeed, let $i:s_X(B)\to X$ and $j:(X,\emptyset)\to (X,s_X(B))$ be the inclusions. Consider the following segment of the long exact sequence in cohomology associated with the pair $(X,s_X(B))$:
$$
\xymatrix{
{H^{*+1}(X)} \ar[r]^{i^*} & {H^{*+1}(s_X(B))} \ar[r]^{\delta ^*} & {H^*(X, s_X(B))} \ar[r]^(.6){j^*} & {H^*(X)} \ar[r]^(.4){i^*} & {H^*(s_X(B)).}
}
$$

Since $i^*$ is clearly surjective in all dimensions, we have that $\delta^*=0$. Consequently, $j^*$ must also be surjective.
This implies that
$\mathrm{ker}(s_X^*)=\mathrm{ker}(i^*)=\mathrm{im}(j^*)\cong H^*(X,s_X(B)).$

All the previous arguments enable us to define the fibrewise pointed cohomology of $X$ in a more convenient form:

\begin{definition}\label{def:fib-cohomology}
Let $X$ be a fibrewise pointed space over $B$. The fibrewise
pointed cohomology of $X$ (with coefficients in $R$) is defined as $H_B^*(X):=H^*(X,s_X(B))$. 
\end{definition}

We now proceed to verify the cohomological lower bound on the fibrewise unpointed LS category of a fibrewise pointed space, as given by James and Morris \cite{J-M}.

\begin{proposition}
Let $X$ be a fibrewise pointed space.  Then 
 $
 nil(H_B^*(X))\leq \ct_B^*(X).
 $
\end{proposition}

\begin{proof}
Since $H_B^*(X)=\mathrm{ker}(s_X)$ , we have 
$nil(H_B^*(X))\leq \sct(s_X)\leq \sct_B(s_X)=\ct_B^*(X).$
\end{proof}

We now aim to establish a similar bound in the context of a fibrewise pointed map $\pi:E\to X$. To do so, we consider the ring homomorphism
$
\pi^*:H_B^*(X)\to H_B^*(E),
$ which is induced in cohomology by the map of pairs $\pi: (E,s_E(B))\to (X,s_X(B))$.

\begin{theorem}\label{thm:coho-lb-secatB}
Let $\pi:E\to X$ be a fibrewise pointed map. Then 
$
nil(\mathrm{ker}(\pi^*))\leq \sct_B^B(\pi).
$    
\end{theorem}

\begin{proof}
Suppose $\sct_B^B(\pi)=n$. Consider an open cover $\{U_i\}_{i=0}^n$ such that $s_X(B)\subset U_i$ for each $i$, and fibrewise pointed maps $\sigma_i:U_i\to E$ such that $\pi\circ \sigma_i\simeq_B^B \mathrm{inc}_{U_i}$. In particular, this implies that for each $i$, we have the following diagram in the category of pairs of spaces, up to homotopy of pairs: 
$$
\xymatrix{
{(U_i, s_X(B))} \ar[rr]^{\mathrm{inc}_{U_i}} \ar[dr]_{\sigma_i} & & {(X, s_X(B))} \\
 & {(E, s_E(B)).} \ar[ur]_{\pi} &
}
$$

Now let $\alpha_0,\dots,\alpha_n$ be cohomology classes from $\mathrm{ker}(\pi^*)\subset H_B^*(X)=H^*(X,s_X(B))$. Fix $\alpha_i\in H_B^{m_i}(X)$ and consider the following portion of the cohomology exact sequence associated with the triple $(X,U_i,s_X(B))$:
$$\xymatrix{
{H^{m_i}(X, U_i)} \ar[r]^{q_i^*} & {H^{m_i}(X, s_X(B))} \ar[r]^{\mathrm{inc}_{U_i}^*} & {H^{m_i}(U_i, s_X(B))} }$$

Since 
$\mathrm{inc}_{U_i}^*(\alpha_i) = \sigma_i^*(\pi^*(\alpha_i)) = \sigma_i^*(0) = 0$, there exist $\bar{\alpha}_i\in H^{m_i}(X,U_i)$ such that $q_i^*(\bar{\alpha}_i)=\alpha_i$.
Finally, we obtain 
$$\begin{array}{ll}
\alpha_0\cup \alpha_1\cup \dots \cup\alpha_n  &  =  q_0^*(\bar{\alpha}_0)\cup q_1^*(\bar{\alpha}_1)\cup\dots q_n^*(\bar{\alpha}_n)\\
& =q^*(\bar{\alpha}_0\cup \bar{\alpha}_1\cup\dots \cup \bar{\alpha}_n)\\
&= q^*(0)=0 \quad (as ~ \bar{\alpha}_0\cup \bar{\alpha}_1\cup\dots \cup \bar{\alpha}_n\in H^m(X,X)),
\end{array}$$
where $q:(X,s_X(B))\to (X,\cup_{i=0}^nU_i)=(X,X)$ denotes the natural inclusion and $m=\sum_{i=0}^nm_i$. 
\end{proof}

We now apply \Cref{prop:DBB-is-secatB} to establish the cohomological lower bound on the parametrized pointed homotopic distance.
Let $f,g:X\to Y$ be fibrewise pointed maps. Recall the pullback diagram described in \eqref{eq:ptd-pullback-pathfib}.
By \Cref{prop:DBB-is-secatB}, we have $D_B^B(f,g)=\sct_B^B(\widetilde{\Pi}_Y)$.
Now, let $\widetilde{\Pi}_Y^*$ be the induced map on cohomology. As a consequence of \Cref{thm:coho-lb-secatB}, we have the following cohomological bound:

\begin{proposition}\label{prop:coho-lb-DBB}
Let $f,g:X\to Y$ be fibrewise pointed maps. Then 
$nil(\mathrm{ker}(\widetilde{\Pi}_Y^*))\leq D_B^B(f,g).$   \end{proposition}

We now establish a cohomological lower bound for the pointed parametrized topological complexity.
\begin{proposition}
Let $X$ be a fibrewise pointed space over $B$ and let $\Delta_X:X\to X\times_B X$ be the diagonal map. Then the following inequality holds: $$nil(\mathrm{ker}[\Delta_X^*:H_B^*(X\times_B X)\to H_B^*(X)])\leq \TC_B^B(X).$$
\end{proposition}

\begin{proof}
Recall that $\TC_B^B(X)=D_B^B(pr_1,pr_2)$, where $pr_1,pr_2:X\times_B X\to X$ are the canonical projections. Since $(pr_1,pr_2):X\times_B X\to X\times_B X$ is the identity map, the pullback $\widetilde{\Pi}_X$ coincides with $\Pi_X$.
Moreover, we have the following commutative diagram up to pointed fibrewise homotopy:
$$\xymatrix{
{X} \ar[rr]^{\gamma _X} \ar[dr]_{\Delta _X} & & {P_B(X)} \ar[dl]^{\Pi _X} \\
 & {X\times_B X,} & 
}$$
where $\gamma _X$ is the  fibrewise pointed homotopy equivalence defined by mapping a point $x$ to $(p_X(x),c_x)$, where $c_x$ represents the constant path at $x$.
Therefore, we have the equality $\mathrm{ker}(\Delta_X^*)=\mathrm{ker}(\Pi_X^*)$.
Using \Cref{thm:coho-lb-secatB}, we conclude the result.
\end{proof}

\section{Pointed vs. unpointed parametrized homotopic distance}\label{sec:pphd-vs-phd}

To conclude our study of parametrized homotopic distance, we compare it with its pointed counterpart in this final section. For fibrewise pointed maps $f, g: X \to Y$, note that the inequality $D_B(f, g) \leq D_B^B(f, g).
$ naturally holds.

The next result proves that, under relatively mild conditions, the gap between the parametrized homotopic distance and its pointed version is less significant than one might initially anticipate. Before presenting this result, we first introduce some preliminary lemmas.  
Recall that a fibrewise well-pointed space is a fibrewise pointed space where the section is a closed fibrewise cofibration.

\begin{lemma}\label{previo}
Let $X$ be a fibrewise pointed space. If $X$ is both a fibrant fibrewise space and an ANR (absolute neighborhood retract) space, then $X$ is fibrewise well-pointed.
\end{lemma}

\begin{proof}
Applying \cite[Corollary A.6 (ii)]{Calcines-fibrewiseTC}, we conclude that $X$ is fibrewise locally equiconnected; that is, the diagonal map $\Delta_X: X \rightarrow X \times_B X$ is a closed fibrewise cofibration. The result then follows directly from Corollary A.9 and Remark A.10 in \cite{Calcines-fibrewiseTC}.
\end{proof}

Our second lemma is a well-known result attributed to Strøm (see the final remark in \cite{strom1971homotopy}) so we omit its proof. For a topological space $X$, let $c : X \to X^I$ denote the map that assigns to each $x \in X$ the constant path at $x$, denoted $c_x$. Recall that a \emph{locally equiconnected space} (LEC space) is a topological space $X$ in which the diagonal map $\Delta : X \to X \times X$ is a closed cofibration. CW-complexes and metrizable spaces serve as examples of LEC spaces.

\begin{lemma}\label{lec}
Let $X$ be a topological space. Then,
\begin{enumerate}
    \item $c:X\rightarrow X^I$ is a closed cofibration if and only if there exists a continuous map $\varphi :X^I\rightarrow I$ such that $c(X)=\varphi ^{-1}(\{0\}).$ 
    
    \item If there is a continuous map $\psi :X\times X\rightarrow I$ such that $\psi ^{-1}(\{0\})=\Delta (X)$, then $c:X\rightarrow X^I$ is a closed cofibration. In particular, if $X$ is a LEC space, $c:X\rightarrow X^I$ is guaranteed to be a closed cofibration.
\end{enumerate}  
\end{lemma}

\begin{remark}
As observed by Str{\o}m, P. Tulley  provides in \cite{tulley1965regularity} an example of a space $X$ that does not admit a continuous map $\varphi :X^I\rightarrow I$ satisfying $\varphi ^{-1}(\{0\})=c(X)$. Equivalently, $c: X \to X^I$ is not a closed cofibration.
\end{remark}

\begin{lemma}\label{previo2}
Let $Y$ be a fibrant fibrewise space over a LEC space $B$. If $Y$ is an ANR space, then $P_B(Y)$ is also an ANR space. Consequently, if $Y$ is additionally a fibrewise pointed space, $P_B(Y)$ must be fibrewise well-pointed.
\end{lemma}

\begin{proof} Recall the pullback diagram that defines $P_B(Y)$.
By Lemma \ref{lec} above, the map $c: B \hookrightarrow B^I$, which assigns to each $b \in B$ the constant path $c_b$, is an ordinary closed cofibration. Since $p^I_Y$ is a Hurewicz fibration, it follows from \cite[Theorem 12]{strom1968note} that $P_B(Y) \hookrightarrow Y^I$ is also a closed cofibration. Given that $Y^I$ is an ANR space, we conclude from \cite[Chapter IV, Theoren 3.2]{hu1965theory} that $P_B(Y)$ is also an ANR space. The final statement follows from the fact that $P_B(Y)$ is fibrant, allowing us to apply Lemma \ref{previo} above.
 \end{proof}

Now, we are in a position to state and prove our result:

\begin{theorem}\label{thm:pointed-unpointed-ineq}
Let $f,g:X\rightarrow Y$ be fibrewise pointed maps where $X$ and $Y$ are fibrant. If, in addition, $X$ and $Y$ are ANR spaces and $B$ is a LEC space, then
$$D_B(f,g)\leq D^B_B(f,g)\leq D_B(f,g)+1.$$
\end{theorem}

\begin{proof}
By Lemmas \ref{previo} and \ref{previo2}, both $X$ and $P_B(Y)$ are fibrewise well-pointed spaces. Additionally, by Proposition \ref{important2}, $Y\times _BY$ is also a fibrewise well-pointed space. Therefore, by Proposition \ref{important2}, $\mathcal{P}_B(f,g)$ is likewise fibrewise well-pointed. Considering $D_B(f,g)=\mbox{secat}_B(\widetilde{\Pi} _Y)$ and $D_B^B(f,g)=\mbox{secat}_B^B(\widetilde{\Pi} _Y)$, the result follows from \cite[Theorem 4.1]{GC}.
\end{proof}

\begin{remark}
Note that \Cref{thm:pointed-unpointed-ineq} also holds under the same hypotheses for $Y$ and  $B$ but requiring that $X$ be metrizable and fibrewise well-pointed, rather than necessarily fibrant or an ANR space.
\end{remark}
As a consequence to \Cref{thm:pointed-unpointed-ineq} we recover \cite[Corollary 25]{Calcines-fibrewiseTC}.

To conclude, we will provide sufficient conditions for the parametrized homotopic distance to coincide with its pointed version. Recall that a fibrewise pointed space $X$ is said to be \emph{cofibrant} if the section $s_X:B\rightarrow X$ is a closed cofibration.

\begin{lemma}\label{lemma-comparison1}
Let $Y$ be a fibrewise pointed space over a LEC space $B$. If $Y$ is both cofibrant and fibrant, then $P_B(Y)$ is also cofibrant (and fibrant).
\end{lemma}

\begin{proof}
We consider the pullback that defines $P_B(Y)$, along with $s=s_{P_B(Y)}:B\rightarrow P_B(Y)$, the induced section on that pullback:
$$\xymatrix{
{B} \ar@{^{(}->}[rr]^c \ar@/_1pc/[ddr]_{id_B} \ar@{.>}[dr]^{s} & & {B^I} \ar@{^{(}->}[d]^{s^I_Y}\\
 & {P_B(Y)} \ar@{^{(}->}[r]_{\overline{c}} \ar[d] & {Y^I} \ar[d]^{p_Y^I} \\
 & B \ar@{^{(}->}[r]_c & {B^I}.
}$$ 
Since $p^I_Y$ is a Hurewicz fibration and $c: B \hookrightarrow B^I$ is a closed cofibration, by applying \cite[Theorem 12]{strom1968note}, we obtain that the base change $\overline{c}: P_B(Y) \hookrightarrow Y^I$ is also a closed cofibration (and the projection $P_B(Y) \rightarrow B$ is a Hurewicz fibration, confirming that $P_B(Y)$ is fibrant). Furthermore, since $s^I_Y$ is a closed cofibration (see \cite[Lemma 4]{strom1971homotopy}), the composition $\overline{c} \circ s = s^I_Y \circ c$ is a closed cofibration. By \cite[Lemma 5]{strom1971homotopy}, we conclude that $s$ is a closed cofibration.
\end{proof}

The following lemma is well-known; for instance, the concluding remark in \cite{strom1971homotopy} implies a proof, although it is not explicitly stated. 
Let $B$ be a space, and consider the category $\mathbf{cof}^B$ of closed cofibrations. That is, the objects are closed cofibrations $s_X:B\rightarrow X$, and the morphisms are the commutative triangles between closed cofibrations.

\begin{lemma}\label{lemma-comparison2}
$\mathbf{cof}^B$ is closed under pullbacks along
morphisms that are Hurewicz fibrations. 
\end{lemma}


\begin{theorem}\label{thm:conditions-implying-phd=pphd}
Let $f, g: X \rightarrow Y$ be fibrewise pointed maps between fibrant and cofibrant fibrewise pointed spaces. Additionally, suppose that $B$ is a CW-complex and that the following conditions hold:
\begin{enumerate}
    \item The projection $p_Y: Y \rightarrow B$ is a $k$-equivalence, for some integer $k \geq 1$;
    \item $\mathrm{dim}(B) < (D_B(f,g) + 1) \cdot k - 1$.
\end{enumerate}
Then, it follows that $D_B(f, g) = D_B^B(f, g)$.
\end{theorem}

\begin{proof}
Since $Y$ is fibrant, the map $\Pi_Y: P_B(Y) \rightarrow Y \times_B Y$ is a Hurewicz fibration, implying that $\widetilde{\Pi}_Y$ is also a Hurewicz fibration. Furthermore, as $X$ is fibrant, $\mathcal{P}_B(f, g)$ is also fibrant.

Next, applying Lemma \ref{lemma-comparison1} and Lemma \ref{lemma-comparison2}, we observe that both $P_B(Y)$ and $Y \times_B Y$ are cofibrant. Consequently, by Lemma \ref{lemma-comparison2} once more, we conclude that $\mathcal{P}_B(f, g)$ is also cofibrant.
Finally, the result follows from applying \cite[Theorem 4.4]{GC}, noting that $D_B(f,g) = \mbox{secat}_B(\widetilde{\Pi}_Y)$ and $D_B^B(f,g) = \mbox{secat}_B^B(\widetilde{\Pi}_Y)$. Moreover, condition (1) in the statement of our theorem ensures that the diagonal map $\Delta_Y: Y \rightarrow Y \times_B Y$ is a $(k-1)$-equivalence, which implies that $\Pi_Y$ is also a $(k-1)$-equivalence. Therefore, $\widetilde{\Pi}_Y$ is a $(k-1)$-equivalence.
\end{proof}

As a consequence to \Cref{thm:conditions-implying-phd=pphd}, we recover \cite[Proposition 28]{Calcines-fibrewiseTC}.

\vspace{.5cm}

\noindent \textbf{Acknowledgement}:
The authors gratefully acknowledge the support of the National Board for Higher Mathematics (NBHM), India through grant 0204/10/(16)/2023/R\&D-II/2789, as well as the support from the Spanish Government under grant PID2023-149804NB-I00.

\bibliographystyle{plain} 
\bibliography{references}

\end{document}